\documentclass{article} %

\usepackage{lipsum}
\usepackage{graphicx}
\usepackage{algcompatible}
\usepackage{amsmath} %
\usepackage{amssymb}  %

\usepackage{fullpage}

\usepackage{amsfonts}
\usepackage{graphicx}
\usepackage{mathdots}
\usepackage{amsopn}

\usepackage{amssymb,amsfonts,amsmath,amsthm}
\usepackage{dsfont}
\let\mathbb=\mathds

\usepackage[english]{babel}
\usepackage{pgfplots}
\usepackage{tikz}
\usepackage{tkz-euclide}

\usepackage{enumitem}
\usepackage[normalem]{ulem}

\definecolor{ao(english)}{rgb}{0.0, 0.5, 0.0}
\usepackage{cite}
\usepackage[colorlinks, citecolor = {ao(english)}, linkcolor = {ao(english)}]{hyperref} 
\usepackage[nameinlink]{cleveref}

\usepackage{tikz}
\usepackage{pgfplots}
\usetikzlibrary{arrows,shapes,trees,calc,positioning,patterns,decorations.pathmorphing,decorations.markings}
\usetikzlibrary{matrix}
\usepgfplotslibrary{groupplots}
\pgfplotsset{compat=newest}

\newtheorem{thm}{Theorem}
\crefname{thm}{Theorem}{Theorems}
\newlist{thmenum}{enumerate}{1} %
\setlist[thmenum]{label=\roman*), ref=\thethm~\roman*)}
\crefalias{thmenumi}{thm} 
\newtheorem{prop}{Proposition}
\crefname{prop}{Proposition}{Propositions}
\newlist{propenum}{enumerate}{1} %
\setlist[propenum]{label=\roman*), ref=\theprop~\roman*)}
\crefalias{propenumi}{prop}

\newtheorem{lem}{Lemma}
\crefname{lem}{Lemma}{Lemmas}
\newlist{lemenum}{enumerate}{1} %
\setlist[lemenum]{label=\roman*), ref=\thelem~\roman*)}
\crefalias{lemenumi}{lem}

\newtheorem{cor}{Corollary}
\crefname{cor}{Corollary}{Corollaries}
\newlist{corenum}{enumerate}{1} %
\setlist[corenum]{label=\roman*), ref=\thecor~\roman*)}
\crefalias{corenumi}{cor} 

\newtheorem{rem}{Remark}
\crefname{rem}{Remark}{Remarks}
\newlist{remenum}{enumerate}{1} %
\setlist[remenum]{label=\roman*), ref=\therem~\roman*)}
\crefalias{remenumi}{rem}

\crefname{example}{Example}{Examples}

\crefname{ass}{Assumption}{Assumption}

\crefname{conj}{Conjecture}{Conjectures}
\newtheorem{defn}{Definition}
\crefname{defn}{Definition}{Definitions}
\newlist{defnenum}{enumerate}{1} %
\setlist[defnenum]{label=\roman*., ref=\thedefn~(\roman*.)}
\crefalias{defnenumi}{defn}

\crefname{prob}{Problem}{Problems}

\newcommand{\Rmn}{\mathbb{R}^{n \times m}}

\newcommand{\Rnn}{\mathbb{R}^{n \times n}}

\newcommand{\rk}{\textnormal{rank}}

\newcommand{\sign}{\textnormal{sign}}

\newcommand{\transp}{\mathsf{T}}
\newcommand{\vari}[1]{\textnormal{S}^{-}[#1]}

\newcommand{\varic}[1]{\textnormal{S}^{-}_c[#1]}
\newcommand{\varizc}[1]{\textnormal{S}^{+}_c[#1]}

\newcommand{\variz}[1]{\textnormal{S}^{+}[#1]}

\newcommand{\sgc}[1]{\textnormal{SC}_{#1}}
\newcommand{\tp}[1]{{\textnormal{TP}_{#1}}}

\newcommand{\ssgc}[1]{\textnormal{SSC}_{#1}}

\newcommand{\Scm}{\textnormal{S}_c^-}
\newcommand{\Scp}{\textnormal{S}_c^+}
\newcommand{\Sm}{\textnormal{S}^-}
\newcommand{\Sp}{\textnormal{S}^+}
\newcommand{\cc}[1]{\textnormal{CC}(#1)}
\newcommand{\scc}[1]{\textnormal{SCC}(#1)}
\newcommand{\pmp}[1]{\textnormal{PMP}(#1)}
\newcommand{\pmu}[1]{\textnormal{PM}(#1)}
\newcommand{\spmu}[1]{\textnormal{SPM}(#1)}

\newcommand{\cp}[1]{\textnormal{CP}(#1)}
\newcommand{\floor}[1]{\lfloor #1 \rfloor}
\newcommand{\ceil}[1]{\lceil #1 \rceil}

\newcommand{\Conv}[1]{\mathcal{S}_{#1}}

\newcommand{\Con}[1]{{\mathcal{C}^{#1}}}
\newcommand{\Obs}[1]{{\mathcal{O}^{#1}}}

\newcommand{\supp}[1]{\text{supp}[{#1}]}

\newcommand{\vd}[1]{\textnormal{VD}_{#1}}
\newcommand{\svd}[1]{\textnormal{SVD}_{#1}}
\newcommand{\vb}[1]{\textnormal{VB}_{#1}}
\newcommand{\svb}[1]{\textnormal{SVB}_{#1}}

\newcommand{\cvd}[1]{\textnormal{CVD}_{#1}}

\newcommand{\cvb}[1]{\textnormal{CVB}_{#1}}
\newcommand{\scvb}[1]{\textnormal{SCVB}_{#1}}
\newcommand{\circmat}[1]{\text{C}^{#1}}
\newcommand{\subscript}[2]{$#1 _ #2$}

\newcommand{\compound}[2]{{#1}_{[#2]}}
\colorlet{FigColor1}{blue}
\colorlet{FigColor2}{red}
\colorlet{FigColor3}{ao(english)}
\colorlet{FigColor4}{orange}
\pgfplotsset{every axis plot/.append style={line width=1.5pt}}
\crefformat{equation}{\textup{#2(#1)#3}}
\crefrangeformat{equation}{\textup{#3(#1)#4--#5(#2)#6}}
\crefmultiformat{equation}{\textup{#2(#1)#3}}{ and \textup{#2(#1)#3}}
{, \textup{#2(#1)#3}}{, and \textup{#2(#1)#3}}
\crefrangemultiformat{equation}{\textup{#3(#1)#4--#5(#2)#6}}%
{ and \textup{#3(#1)#4--#5(#2)#6}}{, \textup{#3(#1)#4--#5(#2)#6}}{, and \textup{#3(#1)#4--#5(#2)#6}}

\Crefformat{equation}{#2Equation~\textup{(#1)}#3}
\Crefrangeformat{equation}{Equations~\textup{#3(#1)#4--#5(#2)#6}}
\Crefmultiformat{equation}{Equations~\textup{#2(#1)#3}}{ and \textup{#2(#1)#3}}
{, \textup{#2(#1)#3}}{, and \textup{#2(#1)#3}}
\Crefrangemultiformat{equation}{Equations~\textup{#3(#1)#4--#5(#2)#6}}%
{ and \textup{#3(#1)#4--#5(#2)#6}}{, \textup{#3(#1)#4--#5(#2)#6}}{, and \textup{#3(#1)#4--#5(#2)#6}}

\crefdefaultlabelformat{#2\textup{#1}#3}

\makeatletter
\AddToHook{cmd/appendix/before}{\def\cref@section@alias{appendix}\def\cref@subsection@alias{appendix}}
\makeatother

\title{On Discrete-Time Periodic Monotonicity Preserving Systems \thanks{This work was conducted while the author was a Jane and Larry Sherman Fellow. It was further supported by the Israel Science Foundation (grant no.2406/22).}}

\author{Christian Grussler\thanks{C. Grussler is with the Stephen B. Klein Faculty of Aerospace Engineering, Technion -- Israel Institute of Technology, 3200003 Haifa, Israel
		{\tt cgrussler@technion.ac.il}}}

\begin{document}

\maketitle

\begin{abstract}
	Two nested classes of discrete-time linear time-invariant systems, which differ by the set of periodic signals that they leave invariant, are studied. The first class preserves the property of periodic monotonicity (period-wise unimodality). The second class is invariant to signals with at most two sign changes per period, and requires that periodic signals with zero sign changes are mapped to the same kind. Tractable characterizations for each system class are derived by the use and extension of total positivity theory via geometric interpretations. Central to our results is the characterization of sequentially convex contours via consecutive minors.  
	
	Our characterizations also extend to the loop gain of Lur'e feedback systems as the considered signals sets are invariant under common static non-linearities, e.g., ideal relay, saturation, sigmoid function, quantizer, etc. The presented developments aim to form a base for future signal-based fixed-point theorems towards the prediction of self-sustained oscillations. Our examples on relay feedback systems indicate how periodic monotonicity preservation gives rise to useful insights towards this goal. 
\end{abstract}

\section{Introduction}
Self-sustained (i.e., unforced) oscillation is a fundamental phenomenon that is observed and utilized across many areas of science and engineering. Examples range from mechanical vibrations \cite{nayfeh2024nonlinear}, over gene regulation and brain rhythms \cite{novak2008design,winfree1980geometry}, to PID-autotuning \cite{aastrom1984automatic} via relay feedback. Theoretical tools to predict such oscillations are commonly based on invariance properties facilitated by Lyapunov functions, the Poincar{\'e}-Bendixon theorem, Poincar{\'e} maps \cite{khalil2002}, their recent extensions based on cones and contraction analysis \cite{weiss2019generalization2,ofir2022sufficient,sanchez2009cones,forni2014differential} or harmonic analysis \cite{khalil2002}. Despite the variety of these existing approaches, predicting self-oscillations still remains often challenging (see, e.g., the widely communicated challenge by Karl J. \AA str\"om to characterize the class of linear systems that exhibit self-oscillations in relay feedback).

In Lur'e feedback systems, i.e., the negative feedback interconnection of a linear time-invariant (LTI) system $G$ with a static non-linearity $\psi: \mathds{R} \to \mathds{R}$ (see~\Cref{fig:lure}),
\begin{figure}
	\centering
	\tikzstyle{neu}=[draw, very thick, align = center,circle]
	\tikzstyle{int}=[draw,minimum width=1cm, minimum height=1cm, very thick, align = center]
	\begin{tikzpicture}[>=latex',circle dotted/.style={dash pattern=on .05mm off 1.2mm,
			line cap=round}]
		\node [coordinate, name=input] {};
		\node [int, right of=input, node distance=2cm] (relay) {$\psi (\cdot)$};
		\node [int, right of=relay, node distance=2cm] (system) {$G$};
		\node [coordinate][right of=system, node distance=2.7cm] (output) {};
		\node [int, below left of=system, node distance=1.5cm, shift={(0cm, -0.5cm)}] (feedback) {$-1$};
		\draw [-] (feedback) -| node[above] {} node[below] {} (input);
		\draw [->] (input) -- node[above] {$u(t)$} (relay);
		\draw [->] (relay) -- node[above] {} (system);
		\draw [-] (system) -- node[above] [name=y] {$[\Conv{g} \psi(u)](t)$} (output);
		\draw [->] (output) |- node[above] {} node[below] {} (feedback);
	\end{tikzpicture}
	
	\caption{Lur'e feedback systems with linear time-invariant system $G$ and static non-linearity $\psi$. \label{fig:lure}}
\end{figure}
the presence of a self-sustained oscillation with period $T$ is equivalent to the existence of a $T$-periodic signal $u^\ast$ satisfying 
\begin{equation}
	u^\ast = -\Conv{g}\psi(u^\ast), \label{eq:fixed_point} 
\end{equation}
where $[\Conv{g}u](t) = (g\ast u)(t)$ denotes the convolution operator resulting from the linear system's impulse response $g$. In other words, the prediction of self-oscillations is a fixed-point problem in the domain of periodic functions, which motivates the establishment of signal-based fixed-point theorems. This work aims at facilitating such developments by characterizing systems that leave certain sets of periodic signals invariant. Fixed-point theorems are commonly based on invariant set assumptions (see, e.g., \cite{granas2003fixed,zeidler1986nonlinear}). Concretely, we consider two subsets of $T$-periodic signals:
\begin{enumerate}[label=(\subscript{S}{{\arabic*}})]
	\item  \emph{Periodically monotone} ($\pmu{T}$) signals, i.e., signals with an unimodal (aka., single-peaked) period. \label{item:signal_set1}
	\item Signals with a variation (i.e, number of sign changes) of at most two per period. \label{item:signal_set2}
\end{enumerate}
This choice of sets resembles the predictions made by the describing function analysis (which approximates $u^\ast$ by a sinusoid) \cite{khalil2002}, as well as several other studies, e.g., on relay feedback systems (see~ \cite{goncalves2001global,rabi2021relay,megretski1996global,chaffey2024amplitude}).

The present work derives tractable characterizations of discrete-time linear time-invariant  (DTLTI) systems that leave either of these sets invariant by the use and extension of total positivity theory \cite{karlin1968total}, i.e., the study of variation bounding/diminishing linear operators -- a framework that has only recently found its application in the systems theory community (see, e.g.,  \cite{grussler2020variation,grussler2020balanced,margaliot2018revisiting,roth2024system,bar2023compound,ofir2022sufficient,marmary2025tractable,alseidi2021discrete}). In the first case, we refer to these systems as \emph{periodic monotonicity preserving} ($\pmp{T}$), where our results establish a complete characterization of all such convolution operators. In the second case, under a mild strictness assumptions, we also find such tractable certificates for the classes of so-called \emph{$2$-periodically variation diminishing} ($\cvd{2}(T)$) systems. Since the set of $\pmu{T}$ signals is included in \ref{item:signal_set2}, it holds that $\cvd{2}(T) \subset \pmp{T}$, 
i.e., greater generality of the signal shape comes at the cost of more stringent requirements on the system. Finally, both signal sets are preserved by several common static non-linearities, e.g., ideal relay, saturation, quantizer, sigmoid function, etc. Thus, in combination with our system classes, this provides a tractable class of Lur'e feedback systems for which our signal sets are left invariant by its \emph{loop gain} $u \mapsto -\Conv{g}\psi(u)$. 

Our results extend the literature in several crucial aspects: firstly, despite the long history of periodic monotonicity preserving kernels (see, e.g., the monograph \cite{karlin1968total} and \cite{ruscheweyh1992preservation}), the discrete-time case has not been considered, yet. Secondly, employing existing total positivity theory \cite{karlin1968total}, one can directly characterize $\cvd{2}(T)$ systems in terms of the \emph{$3$-sign-consistency} of a corresponding cyclic matrix, i.e., all of its minors of order $3$ share the same sign. However, checking all such minors quickly becomes intractable as $T$ grows. This work resolves this issue by showing that it suffices to only check a very small subset of these minors. In particular, these minors can be expressed via the impulse response of a related LTI (compound) system and by that allows us to take advantage of earlier aperiodic investigations \cite{grussler2020variation,grussler2022internally}.
Lastly, contrary to the analysis-based investigations in continuous-time \cite{ruscheweyh1992preservation}, this work provides a complementary perspective by relying on total positivity theory and its geometric interpretations. This viewpoint allows us to derive a tractable sufficient certificate for the verification of $3$-sign consistency from the geometric property of $T$-periodic convex contours, which is central to our system characterizations and establishes a new connection to $\cvd{2}$. We also introduce and study so far neglected strict notions, which differ in how the variation is computed in the event of zero elements. This plays a central role in future applications as seen in an initial studies of self-oscillations relay feedback systems \cite{tong2025selfsustained}. Finally, a shortened version of this work, without proofs and lacking several essential parts, has been presented in \cite{grussler2025characterization}.

The remainder of the paper is organized as follows. After some preliminaries in \Cref{sec:prelim}, we review the theory of cyclic variation bounding and total positivity \Cref{sec:vardim}. Subsequently, we derive our main result on the characterization of sequentially convex contours in \Cref{sec:convex_contour} and provide our convolution operator characterizations in \Cref{sec:cyc_var_bound}. In \Cref{sec:LTI_ex}, these results are applied to relay feedback systems. Conclusions are drawn in \Cref{sec:conc} and most proofs are left to the \cref{sec:appendix}. 

\section{Preliminaries}\label{sec:prelim}
This section introduces notations and concepts that are essential for subsequent discussions and derivations. 

\subsection{Notations}
\subsubsection{Sets}
We write $\mathds{Z}$ for the set of integers and $\mathds{R}$ for the set of reals, with  $\mathds{Z}_{\ge 0}$ and $\mathds{R}_{\ge 0}$ standing for the respective subsets of nonnegative elements. Corresponding notations are also used for subsets starting from non-zero values, strict inequality as well as reversed inequality signs. For $k,l \in \mathds{Z}$, we use $(k:l) = \{k,k+1,\dots,l\}$ if $k \leq l$ and $(k:l) = \{k,k-1,\dots,l\}$ otherwise. The set of increasing $r$-tuples with elements in $(1:n)$ is defined by $
	\mathcal{I}_{n,r} := \{ v = \{v_1,\dots,v_r\} \subset \mathds{N}: 1\leq v_1 < v_2 < \dots < v_r \leq n \}$
where the $i$-th element of this set is a result of \emph{lexicographic ordering}. 

\subsubsection{Sequences}
For a sequence $x: \mathds{Z} \to \mathds{R}$ that fulfills $\sup_{i \in \mathds{Z}} |x(i)| < \infty$ or $\sum_{i \in \mathds{Z}} |x(i)| < \infty$, we write $x \in \ell_{\infty}$ or $x \in \ell_1$, respectively. The \emph{support of $x$} is defined by $\supp{x} := \{t \in \mathds{Z}:\; x(t) \neq 0\}$ and the
\emph{forward difference} of $x$ by $\Delta x(t) := x(t+1) - x(t)$. The restriction of $x$ to $(k:l)$ is denoted by $x(k:l)$. $x$ is called \emph{$T$-periodic}, $T \in \mathds{Z}_{\geq 2}$, if $T$ is the smallest integer (i.e., the fundamental period) such that $x(i) = x(i+T)$ for all $i \in \mathds{Z}$. The set of all $T$-periodic sequences is denoted by $\ell_\infty(T)$. The sequence of all ones is denoted by $\mathbf{1}$ and the vector of all ones in $\mathds{R}^n$ by $\mathbf{1}_n$. 
\subsubsection{Matrices}
For matrices $X = (x_{ij}) \in \Rmn$, we say that $X$ is
\emph{nonnegative},
$X \geq 0$ 
or $X \in \Rmn_{\geq 0}$
if all elements $x_{ij} \in \mathbb{R}_{\geq 0}$. Corresponding notations are also used for matrices with strictly positive entries, reversed inequality signs. %
If $X\in \Rnn$, then $\sigma(X) = \{\lambda_1(X),\dots,\lambda_n(X)\}$ denotes its \emph{spectrum}, where the eigenvalues are ordered by descending absolute value, i.e., $\lambda_1(X)$ is the eigenvalue with the largest magnitude, counting multiplicity. If the magnitude of two eigenvalues coincides, we sub-sort them by decreasing real part. %
A \emph{(consecutive) $j$-minor} of $X \in \Rmn$ is a minor that is constructed of (consecutive) $j$ columns and $j$ rows of $X$. The submatrix with rows $\mathcal{I} \subset (1:n)$ and columns $\mathcal{J} \subset (1:m)$ is written as $X_{(\mathcal{I},\mathcal{J})}$, where we also use the notions
$X_{(:,\mathcal{J})} := X_{((1:n),\mathcal{J})}$, $X_{(\mathcal{I},:)} := X_{(\mathcal{I},(1:m))}$. In case of subvectors, we simply write $x_\mathcal{I}$. 

We define the $(i,j)$-th entry of the so-called \emph{$r$-th multiplicative compound matrix} 
$\compound{X}{r} \in \mathbb{R}^{\binom{n}{r} \times \binom{m}{r}}$ of 
$X \in \Rmn$ by 
$\det(X_{(\mathcal{I},\mathcal{J})})$,
where $\mathcal{I}$ is the $i$-th and $\mathcal{J}$ is the $j$-th element in $\mathcal{I}_{n,r}$ and $\mathcal{I}_{m,r}$, respectively. 
For example, if $X \in \mathbb{R}^{3 \times 3}$, then
\begin{align*}
\compound{X}{2} = \begin{bmatrix}
		  \det(X_{(\{1,2\},\{1,2\})}) & \det(X_{(\{1,2 \},\{1,3\})}) & \det(X_{(\{1,2 \},\{2,3\})})\\
		 \det(X_{(\{1,3 \},\{1,2 \})}) & \det(X_{(\{1,3 \},\{1,3\})}) & \det(X_{(\{1,3 \},\{2,3\})})\\
		 \det(X_{(\{2,3 \},\{1,2 \})}) & \det(X_{(\{2,3 \},\{1,3\})}) & \det(X_{(\{2,3 \},\{2,3\})})\\
		 \end{bmatrix}. \normalsize
\end{align*}
The following properties of the multiplicative compound matrix will be elementary to our discussion (see, e.g., \cite[Section~6]{fiedler2008special} and \cite[Subsection~0.8.1]{horn2012matrix}).
\begin{lem}\label{lem:compound_mat}
	Let $X \in \mathbb{R}^{n \times p}$, $Y \in \mathbb{R}^{p \times m}$ and $1 \leq r  \leq \min \{n,m,p \}$.
	\begin{lemenum}
		\item $\compound{(XY)}{r} = \compound{X}{r}\compound{Y}{r}$ (Cauchy-Binet formula). \label{eq:Cauchy_Binet}
		\item $\sigma(\compound{X}{r}) = \{\prod_{i \in I} \lambda_i(X): I \in \mathcal{I}_{n,r} \}$.
		\item $\compound{X^\transp}{r} = (\compound{X}{r})^\transp$.
	\end{lemenum} 
\end{lem}

The \emph{cyclic matrix} corresponding to $g \in \ell_\infty(T)$ is defined by 
\begin{equation*}
	\circmat{g} := \begin{bmatrix}
		g(0) & g(T-1) & \dots & g(1) \\
		g(1) & g(0) & \dots & g(2)\\
		\vdots & \vdots & &\vdots\\
		g(T-1) & g(T-2) & \dots & g(0)
	\end{bmatrix}
\end{equation*}

\subsubsection{Functions}
The indicator function of a subset $\mathcal{S} \subset \mathds{R}$ is denoted by 
\begin{equation*}
\mathbb{1}_\mathcal{S}(x) := \begin{cases}
	1 & x \in \mathcal{S}\\
	0 & x \not \in \mathcal{S},
\end{cases}	
\end{equation*}
the \emph{Heaviside function} by $s(t) := \mathds{1}_{\mathbb{R}_{\geq 0}}(t)$, the \emph{unit pulse function} by $\delta(t):=\mathds{1}_{\{0\}}(t)$ and the \emph{sign function} by $\sign(t) := \mathds{1}_{\mathbb{R}_{> 0}}(t) - \mathds{1}_{\mathbb{R}_{< 0}}(t)$. Further, for any $x \in \mathbb{R}$, we define the \emph{floor} of $x$ as $\lfloor x \rfloor = \max \{ n \in \mathbb{Z} : n \leq x \}$ and the \emph{ceiling} of $x$ as $\lceil x \rceil = \min \{ n \in \mathbb{Z} : n \geq x \}$.
\subsection{Linear Dynamic Systems}
Discrete-time linear time-invariant (DTLTI) systems with input signal $u \in \ell_\infty$ and output signal $y \in \ell_{\infty}$ are described by the \emph{impulse-response} $g \in \ell_1$, i.e., the output of $u = \delta$ and the \emph{convolution operator} 
\begin{align}
	\label{eq:def_conv}
	y(t) = [\mathcal{S}_gu](t) := 	(g \ast u)(t) := \sum_{\tau = -\infty}^\infty g(t-\tau) u(\tau), \;	t \in \mathds{Z},
\end{align}
where in case of $u \in \ell_{\infty}(T)$, this also reads as
\begin{equation}
	\begin{aligned}
			[\Conv{g} u](t) &= \sum_{\tau =-\infty}^\infty g(t-\tau)u(\tau) = \sum_{k = -\infty}^\infty \sum_{\tau = kT}^{(k+1)T-1} g(t-\tau)u(\tau)\\
		&=  \sum_{k = -\infty}^\infty \sum_{\tau = 0}^{T-1} g(t-kT-\tau)u(kT+\tau) = \sum_{\tau = 0}^{T-1} g_T(t-\tau)u(\tau)
	\end{aligned} \label{eq:periodic_conv}
\end{equation}
with
\begin{align}
	g_T(t) := \sum_{k=-\infty}^{\infty}g(t-kT) \in \ell_{\infty}(T). \label{eq:periodic_sum}
\end{align}
denoting the \emph{periodic summation of $g$}. Therefore, if $g \in \ell_\infty(T)$, we also define %
\begin{align}
		\label{eq:def_conv_cyc}
[\Conv{g}u](t) := (g \ast u)(t) :=\sum_{\tau = 0}^{T-1} g(t-\tau) u(\tau), \;	t \in (0:T-1),
\end{align} 
for $u \in \ell_\infty(T)$, which can equivalently be identified with
\begin{equation}
	y(0:T-1) = \circmat{g} u(0:T-1). \label{eq:cyc_mat_u_y}
\end{equation}
This also implies the following simple composition rules:
\begin{lem}\label{lem:psum_compo}
	Let $g_1,g_2 \in \ell_1$. Then, it holds that
	\begin{lemenum}
		\item $(g_1 \ast g_2)_T = {g_1}_T \ast {g_2}_T$
		\item $(g_1 + g_2)_T =  {g_1}_T + {g_2}_T$
	\end{lemenum}
\end{lem}
If $g \in \ell_{\infty}$ is the  impulse response to a causal DTLTI system, then 
\begin{equation*}
	G(z) = \sum_{k=0}^{\infty} g(k)z^{-k} = \frac{r \prod_{i=1}^{m} (z-z_i)}{\prod_{j=1}^{n}(z-p_i)}
\end{equation*}
is its \emph{transfer function}, where $r \in \mathbb{R}$, and $p_i$ and $z_i$ are referred to as \emph{poles} and \emph{zeros}, and $n -m$ as the \emph{relative degree}. We assume that all $p_i$ and $z_i$ are distinct. The tuple $(A,b,c,d)$ is called a \emph{state-space realization} of $G(z)$ if 
\begin{equation}
	\begin{aligned}
		x(t+1) &= Ax(t) + bu(t)\\
		y(t) &= cx(t) + d u(t)
	\end{aligned} \label{eq:ss}
\end{equation}
$A \in \mathds{R}^{N \times N}$, $b,c^\transp \in \mathds{R}^N$, $d \in \mathds{R}$ and $|\lambda_1(A)| < 1$ (asymptotic stability) with $g(t) = cA^{t-1}b s(t-1) + d\delta(t)$. If the relative degree of $G(z)$ is non-zero, then $d =0$ and we simply write $(A,b,c)$. 
For $t,j \in \mathds{Z}_{>0}$, we will make use of the \emph{Hankel matrix}   
\begin{subequations} 
	\begin{align}
		H_g(t,j) &:= \begin{pmatrix}
			g(t) & g(t+1) & \dots & g(t+j-1)\\
			g(t+1) & g(t+2) & \dots & g(t+j)\\
			\vdots & \vdots & \ddots   & \vdots \\
			g(t+j-1)   & g(t+j) & \dots & g(t+2(j-1))\\
		\end{pmatrix} = \Obs{j}(A,c) A^{t-1} \Con{j}(A,b) \label{eq:hankel_mat_decomp}
	\end{align}
	where 
	\begin{align}
		\label{eq:truncated_con}
		\Con{j}(A,b) &:= \begin{pmatrix}
			b & Ab & \dots & A^{j-1}b
		\end{pmatrix}	\\
		\label{eq:truncated_obs}
		\Obs{j}(A,c) &:= \Con{j}(A^\transp,c^\transp)^\transp.
	\end{align}
\end{subequations}
$(A,b,c,d)$ is a \emph{minimal realization} of $G(z)$, if $A \in \mathds{R}^{n \times n}$ or equivalently, by  \cite[Theorem~ 3.16]{zhou1996robust} $\rk(\mathcal{O}^n(A,c)) = \rk(\mathcal{C}^n(A,b)) = n$. Further, we call an LTI system \emph{(strictly) externally positive}, if $g(t) (>) \geq 0$ for all $t \geq \min\{\supp{g}\}$ and \emph{(strictly) externally negative}, if $-g$ is (strictly) externally positive.

Finally, the periodic summation of $g$ computes in terms of $(A,b,c,d)$ as  
\begin{equation}
	\begin{aligned}
		g_T(t) &= \sum_{k= 0}^\infty g(t+kT) = (d  + \sum_{k=1}^\infty cA^{kT}b) \delta(t) + cA^{t-1} \sum_{k=0}^\infty cA^{kT} b s(t-1) \\ 
		&= (d + cA^{T-1}(I_n-A^T)^{-1}b)\delta(t) + cA^{t-1}(I_n-A^T)^{-1}b s(t-1), 
		\; t \in (0:T-1),
	\end{aligned}\label{eq:g_T_ss}  
\end{equation}
where the last equation results from the well-known \emph{Neumann series}. In particular, if $g(t) = 0$ for all $t \leq l$,  then $g_T(0:T+l-1) = \bar{g}(0:T+l-1)$, where $\bar{g}$ is the impulse response of the system
\begin{equation}
	(\bar{A},\bar{b},\bar{c},\bar{d}) := (A,(I_n-A^T)^{-1}b,c,d + cA^{T-1}(I_n-A^T)^{-1}b). \label{eq:ss_psum}
\end{equation}

\section{Cyclic Variation Bounding Properties}
\label{sec:vardim}
The goal of the present study is to provide tractable characterization of convolution operators with cyclic variation diminishing/bounding properties. In order to define and study these, we will need the following notions and background results. 

\subsection{Variation Bounding} 
\begin{defn}[Variation]
	The \emph{variation} of a vector $u \in \mathds{R}^n$ is defined as the number of sign-changes in $u$, i.e., $\vari{u} := \sum_{i = 1}^{m-1} \mathbb{1}_{\mathbb{R}_{< 0}}(\tilde{u}_i \tilde{u}_{i+1})$, $\vari{0} := -1$, where $\tilde{u} \in \mathds{R}^m$ is the vector that results from deleting all zeros in $u$. Further, we define the \emph{strict variation} as $\variz{u} := \limsup_{p \to u} \vari{p}$, i.e., the maximum number of possible sign changes in $u$ by allowing each zero in $u$ to be replaced by $\pm 1$. The same notations are also used for sequences $u \in \ell_{\infty}$. 
\end{defn}
Although $\vari{u} \leq \variz{u}$, equality does not necessarily hold, e.g., $\vari{\begin{pmatrix}
		1 & 0 & 2
\end{pmatrix}} = \vari{\begin{pmatrix}
		1 & 2
\end{pmatrix}} = 0$, but $\variz{\begin{pmatrix}
		1 & 0 & 2
\end{pmatrix}} = \vari{\begin{pmatrix}
		1 & -1 & 2
\end{pmatrix}}  = 2$.
In particular, the following limit inequalities hold \cite[Lemma~5.1.1]{karlin1968total}:  
\begin{lem}\label{lem:lim_var}
	Let $x: \mathds{R} \to \mathds{R}^n$ be such that $\lim_{t \to \infty} x(t) = u \in \mathds{R}^n$. Then, $\vari{u} \leq \liminf_{t \to \infty} \variz{x(t)}$ and $\limsup_{t \to \infty} \vari{x(t)} \leq \variz{u}$. 
\end{lem}

Based on these definitions, we can also define their so-called cyclic analogues. 
\begin{defn}[Cyclic Variation]
	For a vector $u \in \mathds{R}^n$, its \emph{cyclic variation} is defined by $$
		\varic{u} := \sup_{i \in (1:n)}  \vari{\begin{pmatrix}
				u_i & u_{i+1} & \dots & u_n & u_1 &\dots & u_i
		\end{pmatrix}}$$
	and its \emph{strict cyclic variation} by $
		\varizc{u} := \sup_{i \in (1:n)}  \variz{\begin{pmatrix}
				u_i & u_{i+1} & \dots & u_n & u_1 &\dots & u_i
		\end{pmatrix}}$
\end{defn}
It is important to note that $\varic{u}$ and $\varizc{u}$ are always \emph{even} for $u \neq 0$, because
\begin{align}
	\varic{u} &= \begin{cases}
		\vari{u} +1 & \text{if } \vari{u} \text{ is odd},\\
		\vari{u} & \text{if } \vari{u} \text{ is even},
	\end{cases} \label{eq:even_var}
\end{align}
and analogously for $\varizc{u}$. With slight abuse of notation, we also apply these definitions to $u \in \ell_\infty(T)$ and mean that we evaluate the cyclic variations of $u(0:T-1)$, which via \cref{eq:even_var} is uniquely determined by $\vari{u(0:T-1)}$ and $\variz{u(0:T-1)}$, respectively. 

\subsubsection{Variation Diminishing \& Bounding}
We are ready, now, to define the notions of variation diminishing/bounding linear mappings.  
\begin{defn}[$k$-variation bounding/diminishing]
	\label{def:vb_k}
	For $X \in \Rmn$ and $k \in (0:m-1)$, the linear map $u \mapsto Xu$ is said to be 
\begin{defnenum}
	\item  \emph{(strictly) $k$-variation bounding} ($\textnormal{(S)}\vb{k}$), if for all $u\neq 0$ with $\vari{u} \leq k$ it holds that $\vari{Xu} \leq k$ ($\variz{Xu} \leq k$).
\item \emph{(strictly) $k$-variation diminishing} ($\textnormal{(S)}\vd{k}$), if $X$ is $\textnormal{(S)}\vb{j}$ for all $j \in (0:k)$.
\end{defnenum}
	The same notions are also used for operators $X: \ell_{\infty} \to \ell_\infty$ and $k \in \mathds{Z}_{\geq 0}$. For brevity, we write $X \in \textnormal{(S)}\vb{k}$ and $X \in \textnormal{(S)}\vd{k}$, respectively.
\end{defn}
By replacing the variations in \cref{def:vb_k} with their cyclic counter-parts, we also get the corresponding definitions of \emph{(strictly) cyclic $2k$-variation bounding} $(\text{(S)}\cvb{2k})$ and \emph{(strictly) cyclic $2k$-variation diminishing} $(\textnormal{(S)}\cvd{2k})$, where in the operator case, we consider $X: \ell_{\infty}(T) \to \ell_{\infty}(T)$. The sufficiency of only considering even cyclic cases is due to \cref{eq:even_var}, which also yields the following direct equivalences:
\begin{lem}\label{lem:equiv_cvb_cvd}
	Let $X \in \Rmn$, $u \in \mathds{R}^m$ and $k \in \mathds{Z}_{\geq 0}$. Then, the following hold:
	\begin{lemenum}
		\item $\vari{u} \leq 2k$ $(\variz{u}) \leq 2k)$ if and only if $\varic{u}\leq 2k$ $(\varizc{u}) \leq 2k)$. \label{item:var_varc} 
		\item $X \in \textnormal{(S)}\cvd{2k}$ if and only if $X \in \textnormal{(S)}\cvb{2j}$ for all $j = (0:k)$. \label{item:cvd_2k}
		\item $X \in \textnormal{(S)}\cvb{2k}$ if and only if $X \in \textnormal{(S)}\vb{2k}$. \label{item:cvb_vb}
	\end{lemenum}
\end{lem}
\subsubsection{Periodic Monotonicity Preservation}
We further consider the shift invariant relaxation of $\textnormal{(S)}\vb{2}$ of preserving \emph{periodic monotonicity}, where the latter notion is defined as follows: \begin{defn}[Periodic Monotonicity]
	\label{def:pm} 
	$u \in \ell_{\infty}(T)$ is called (strictly) $T$-periodically monotone ($u \in \textnormal{(S)}\pmu{T}$) if $\varic{u-\gamma \mathbf{1}} \leq 2$ ($\varizc{u-\gamma \mathbf{1}} \leq 2$) for all $\gamma \in \mathds{R}$.
\end{defn}
In other words, $u \in \pmu{T}$ cannot cross any horizontal line more than twice in a period or equivalently, there exist $t_1 \le t_2 \le t_1+T-1$ such that 
\begin{equation}
	u(t_1) \leq u(t_1+1) \leq \dots \leq u(t_2) \geq u(t_2+1) \geq \dots \geq u(t_1+T-1), \label{eq:unimod_perido}
\end{equation}
i.e., $u$ has a so-called \emph{unimodal (single-peaked)} period (with peak at $u(t_2)$). If $u \in \spmu{T}$, it is additionally required that there is at most one $t \in (t_1:t_1+T-1)$ with $\Delta u(t) = 0$, which then has to fulfill $\sign(\Delta u(t-1)) = -\sign(\Delta u(t+1)) \neq 0$, i.e., \cref{eq:unimod_perido} holds with strict inequalities with the only possible exception $u(t_2) = u(t_2+1)$.
\begin{figure}
	\centering
	\begin{tikzpicture}
		\begin{axis}[height=3.5 cm,
			width=12 cm,
			xmax=10,xmin = -10,
			xtick=\empty, ytick = {0}, yticklabels={$\gamma$},
			axis line style = { draw = none }, ymin = -2, ymax = 4,
			]
			\addplot+[ycomb,black,thick,samples at={-10,-9,...,10}]{cos(deg(2*pi/8*x))};
			\addplot[color = red,thick] coordinates{(-4,-2) (-4,2)};
			\addplot[color = red,thick] coordinates{(4,-2) (4,2)};
			\addplot[color = red,thick,dashed] coordinates{(0,-2) (0,2)};
			\addplot[color = black] coordinates{(-10,0) (10,0)};		
			\node[above] at (axis cs:-4,2) {$u(t_1)$};		
			\node[above] at (axis cs:4,2) {$u(t_1+T)$};		
			\node[above] at (axis cs:0,2) {$u(t_2)$};		
		\end{axis}
	\end{tikzpicture}
	\caption{Example of a periodically monotone (single-peaked) $u \in \ell_\infty(T)$: $\varic{u - \gamma \mathbf{1}} \leq 2$ for all $\gamma \in \mathds{R}$, or equivalently, there exist $t_1 \le t_2 \le t_1+T-1$  such that \cref{eq:unimod_perido} is fulfilled. In other words, over the period $(t_1:t_1+T-1)$, every local maximum $u(t)$ is a global maximum (single peak), or simply $\varic{\Delta u} \leq 2$. In particular, $u$ is also strictly periodically monotone, because $\varizc{\Delta u} \leq 2$. 
		\label{fig:single-peaked}}
\end{figure}
Thus, in our derivations, we can use the following equivalent characterization:
\begin{lem} \label{lem:pmu_delta}
	Let $u \in \ell_{\infty}(T)$, then $u \in \textnormal{(S)}\pmu{T}$ if and only if  $\varic{\Delta u} \leq 2$ ($\varizc{\Delta u} \leq 2$). 
\end{lem}
An illustration of these notions and \cref{lem:pmu_delta} is given in 
\cref{fig:single-peaked}. We will also require the following simple composition result, which is proven in \Cref{proof:lem:pmu_mult}:
\begin{lem}\label{lem:pmu_mult}
		Let $u \in \textnormal{(S)}\pmu{T}$ and $\tilde{u}(t) := {u}(t){u}(t+1)$. Then, $\tilde{u} \in  \textnormal{(S)}\pmu{T}$. 
\end{lem}
Finally, let us define \emph{periodic monotonicity preservation} for operators. %
\begin{defn}[Periodic Monotonicity Preservation]
	\label{def:pmp}
	$X: \ell_\infty(T) \to \ell_\infty(T)$, $T \geq 4$, is said to be \emph{(strictly) $T$-periodic monotonicity preserving} ($\textnormal{(S)}\pmp{T}$) if $Xu \in \textnormal{(S)}\pmu{T}$ for all $u \in \pmu{T}$ with $\Delta u \not \equiv 0$. For brevity, we write $X \in \textnormal{(S)}\pmp{T}$. 
\end{defn}
Note that the cases of $T \in \{2,3\}$ are not of interest, since the (strict) cyclic variation of vectors with dimensions smaller or equal to three cannot exceed two.

\subsection{Total Positivity}
Investigations of linear variation bounding/diminishing mappings have a long history in areas such as interpolation theory, complex analysis, statistics, mechanics, etc. (see, e.g., the monographs \cite{karlin1968total,fallat2017total}). Within the systems and control community, however, such investigations have just begun (see, e.g.,  \cite{grussler2020variation,grussler2020balanced,margaliot2018revisiting,roth2024system,bar2023compound,ofir2022sufficient,marmary2025tractable,alseidi2021discrete}). Central to the characterization of such mappings is the mathematical framework of \emph{total positivity} \cite{karlin1968total}. Introduction of this framework requires the following matrix notions:
\begin{defn}[$k$-Sign Consistency \& $k$-Positivity]\label{def:k_sc_matrix}
	Let $X \in \Rmn$ and $k \leq \min\{m,n\}$. $X$ is called 
\begin{defnenum}
		\item \emph{(strictly) $k$-sign consistent} ($\textnormal{(S)}\sgc{k}$) if
	$\compound{X}{k} (>)\ge 0$ or $\compound{X}{k} \le (<)0$. For brevity, we write $X \in \textnormal{(S)}\sgc{k}$. 
	\item  \emph{(strictly) $k$-positive} ($\textnormal{(S)}\tp{k}$) if
	$\compound{X}{j} (>) \ge 0$ for all $j \in (1:k)$. For brevity, we write $X \in \textnormal{(S)}\tp{k}$ and in case that $k = \min\{m,n\}$ we say that $X$ is \emph{(strictly) totally positive}.
\end{defnenum}
\end{defn}
Note that $X \in \Rmn$ is $0$-variation bounding/diminishing if and only if $X = \compound{X}{1} \in \sgc{0}$, i.e., all entries share the same sign. This characterization also generalizes as follows \cite{roth2024system,karlin1968total}:
\begin{prop}\label{prop:sc_k_mat_vb_m}
	For $X \in \Rmn$, $n \geq m$, the following hold:
	\begin{propenum}
		\item if $k \leq \min\{m,n\}-1$, then $X \in \svb{k-1}$ if and only if $X \in \ssgc{k}$.
		\item if $\rk(X) = m < n$, then $X \in \vb{m-1}$ if and only if $X \in \sgc{m}$. \label{item:SC_full_rank}
		\item if $r = \rk(X) < m$, then $\vari{Xu} \leq \rk(X)-1$ for all $u \in \mathds{R}^m$ if and only if each column of $\compound{X}{r}$ is $\sgc{0}$. \label{item:SC_low_rank}
		\item if $k < \rk(X)$ and any $k$ columns of $X$ are linearly independent, then $X \in \vb{k-1}$ if and only if $X \in \sgc{k}$.  \label{item:SC_lin_ind}
	\end{propenum}
\end{prop}
By \cref{item:cvb_vb}, analogous cyclic characterizations then read:
\begin{cor}\label{cor:cvb_k}
	Let $X \in \Rmn$, $n \geq m$. Then, the following hold:
	\begin{corenum}
		\item if $2k+2 \leq \min\{m,n\}$, then $X \in \scvb{2k}$ if and only if $X \in \ssgc{2k+1}$. \label{item:SCVB_SSC}
		\item if $\rk(X) = m = 2k+1 < n$, then $X \in \cvb{2k}$ if and only if $X \in \sgc{2k+1}$. \label{item:CVB_SC}
		\item if $2k+1 < \rk(X)$ and any $2k+1$ columns of $X$ are linearly independent, then $X \in \cvb{2k}$ if and only if $X \in \sgc{2k+1}$. \label{item:CVB_lin_ind}
	\end{corenum}
\end{cor}
Verifying that $X \in \textnormal{(S)}\sgc{m}$ requires a tractable way of checking whether all elements in $\compound{X}{m}$ share the same (strict) sign. Instead of computing all entries in $\compound{X}{m}$, certain strictness assumptions make it sufficient to only consider consecutive minors \cite[Theorem~3.3.1]{karlin1968total}:
\begin{prop} \label{prop:karlin_scm}
	Let $X \in \mathds{R}^{n \times m}$, $n\geq m$, be such that 
	\begin{propenum}
		\item all consecutive $j$-minors of $X_{(:,(m:m-j+1))}$ share the same strict sign for all $j \in (1:m-1)$.
		\item all consecutive $m$-minors of $X$ share the same (strict) sign.
	\end{propenum}
	Then, $X$ is $\textnormal{(S)}\sgc{m}$. 
\end{prop}
\cref{prop:karlin_scm} can also be used to prove the following sufficient condition for $k$-positivity \cite[Proposition~8]{grussler2020variation}: 
\begin{cor} \label{cor:k_pos_consec_minors}
	Let $X \in \Rmn$, $k \leq \min \{n,m\}$, be such that
	\begin{corenum}
	\item all consecutive $j$-minors of $X$ are strictly positive, $j \in (1:k-1)$,
	\item all consecutive $k$-minors of $X$ are
	nonnegative (strictly positive).
	\end{corenum}
	Then, $X \in \textnormal{(S)}\tp{k}$, which in the strict case is also a necessary condition. 
\end{cor}

\section{Convex Contours}
\label{sec:convex_contour}
In this section, we present our first main result on the tractable characterization of so-called \emph{sequentially convex contours}. This is a crucial intermediate result in preparation for our certificates of $\Conv{g}$ belonging to the classes $\textnormal{(S)}\pmp{T}$ and $\textnormal{(S)}\cvd{T}$. In particular, our result establishes a natural generalization of \cref{prop:karlin_scm} in case of $m=3$.

\subsection{Monotonicity \& Convexity}
We begin by showing that for $m = 2$ and $m=3$, \cref{prop:karlin_scm} follows from the locality properties of monotonicity and convexity, respectively. Our characterization of convex contours will then follow from these geometric interpretations.
\begin{defn}[Monotonicity]
$f: \mathcal{D} \subset \mathds{R} \to \mathds{R}$ is called (strictly) monotonically decreasing if $f(t_{i_1}) (>) \geq  f(t_{i_2})$ for all $t_{i_1} < t_{i_2}$ from $\mathcal{D}$. $f$ is (strictly) monotonically increasing if $-f$ is (strictly) monotonically decreasing. %
\end{defn}
The following characterization of monotonicity by sign consistency follows from direct computations:
\begin{lem}\label{lem:mono}
  	Let $f: \{t_{1}, t_{2}, \dots, t_{n} \} \to \mathds{R}$ with $t_{i} < t_{i+1}$ for all $i \in (1:n-1)$ and 
  	\begin{equation*}
  		M^f := 	\begin{bmatrix}
  			f(t_{1}) & 1\\
  			f(t_{2}) & 1 \\
  			\vdots & \vdots \\
  		 f(t_n) & 1
  		\end{bmatrix}
  	\end{equation*}
  	Then, the following are equivalent:
\begin{lemenum}
  		\item $f$ is (strictly) monotonically decreasing. 
  		\item $M^f_{[2]} (>) \geq 0$. 
  		\item For all $i \in (1:n-1)$ it holds that \label{item:mono_delta}
  		\begin{equation*}
  		 \det \begin{bmatrix}
  				f(t_i) & 1\\
  				f(t_{i+1}) & 1
  			\end{bmatrix} (>) \geq 0
  		\end{equation*}
\end{lemenum}
\end{lem}
Alternatively, \cref{lem:mono} can also be verified by noticing that a function is monotone if and only if it crosses any value in $\mathds{R}$ at most once and cannot remain in a value under the additional strictness assumption. As this property is invariant with respect to scaling and translation, this translates to
\begin{equation*}
\forall u \in \mathds{R}^2\setminus \{0\}:	(\variz{u_1 f + u_2 \mathbf{1}})\; \vari{u_1 f + u_2 \mathbf{1}}\leq 1,
\end{equation*} 
which by \cref{item:SC_full_rank,item:SC_low_rank} leads to the same conclusion. Moreover, \cref{prop:karlin_scm} in case of $m=2$ can be recovered from \cref{lem:mono} by noticing that 
if the second column of $X \in \mathds{R}^{n \times 2}$ consists of elements with the same strict sign, then by \cref{eq:Cauchy_Binet}
 \begin{equation}
 	X = \begin{bmatrix}
 		x_{12} &  &\\
 		& \ddots & \\
 		& & x_{n2}
 	\end{bmatrix} \begin{bmatrix}
 	\frac{x_{11}}{x_{12}} & 1\\
 	\vdots & \vdots \\
 	\frac{x_{n1}}{x_{n2}} & 1
 \end{bmatrix} \in \text{(S)SC}_2 \label{eq:fact}
 \end{equation}
 if and only if the second factor is $\text{(S)SC}_2$. Interpreting this second factor as the matrix $M^f$ in \cref{lem:mono}, it follows from \cref{item:mono_delta} that it is enough to verify that only consecutive $2$-minors share the same sign. This, however, is equivalent to checking the consecutive $2$-minors of $X$ itself. Next, we will move on to the case of $m=3$ and its links to convexity. 
 \begin{defn}[Convexity/Concavity]\label{def:cvx}
	Let $\mathcal{D} := \{t_1,\dots,t_n\}$ with with $t_{i} < t_{i+1}$ for all $i \in (1:n-1)$. Then, $f: \mathcal{D} \to \mathds{R}$ is \emph{(strictly) convex} if for all fixed $\tau \in \mathcal{D}$
	\begin{equation}
s(t,\tau) := \frac{f(t) - f(\tau)}{t-\tau} = \frac{f(\tau) - f(t)}{\tau-t} 
	 \label{eq:def_cvx_seq}
	\end{equation}
	is (strictly) monotonically increasing in $t$ on $\mathcal{D} \setminus \{\tau\}$. $f$ is called \emph{(strictly) concave} if $-f$ is (strictly) convex. 
\end{defn}
Our definition of convexity via the so-called \emph{slope function} $s(t,\tau)$ is a well-known equivalence in the theory of convex functions \cite[Proposition~1.1.4]{hiriart2013convex} on $\mathds{R}$. It follows by \cite[Proposition~5.3.1]{hiriart2013convex} that the \emph{piecewise-linear continuous extension} $\bar{f}: [t_1,t_n] \to \mathds{R}$ of $f$ defined by
\begin{equation}
	\bar{f}(t) := f(t_i) + s(t_{i+1},t_i)(t-t_i), \; t_i \leq t \leq t_{i+1} \label{eq:f_ext}
\end{equation}
is convex if and only if its right-derivatives are monotonically increasing -- by the piecewise-linearity, these derivatives correspond to $s(t_{i+1},t_i)$. Expressing the monotonicity of the slope function $s(t,\tau)$ via \cref{lem:mono} gives the following characterizations in terms of sign-consistency, which is proven in \Cref{proof:lem:convex}:

\begin{lem}\label{lem:convex}
	Let $f: \{t_{1}, t_{2}, \dots, t_{n} \} \to \mathds{R}$ with $t_{i} < t_{i+1}$ for all $i \in (1:n-1)$ and 
	\begin{equation*}
		M^f := 	\begin{bmatrix}
			t_{1}  & f(t_{1}) & 1 \\
			t_{2}  & f(t_{2}) & 1 \\
			\vdots & \vdots & \vdots \\
			t_{n}  & f(t_{n}) & 1 \\
		\end{bmatrix}
	\end{equation*}
	Then, the following are equivalent:
	\begin{lemenum}
		\item $f$ is (strictly) convex. \label{item:cvx_f}
		\item $\bar{f}$ in \cref{eq:f_ext} is convex (and no three points of $\{(t_i,f(t_i)): t_i \in \mathcal{D}\}$ are co-linear).  \label{item:bar_f}
		\item $M^f_{[3]} (>) \geq 0$.  \label{item:cvx_comp}
		\item For all $i \in (1:n-2)$ it holds that \label{item:cvx_det}
		\begin{equation*}
			 \det \begin{bmatrix}
				t_{i}  & f(t_{i}) & 1 \\
			t_{i+1}  & f(t_{i+1}) & 1 \\
			t_{i+2}  & f(t_{i+2}) & 1	
			\end{bmatrix} (>) \geq 0.
		\end{equation*}
	\end{lemenum}
\end{lem}
Using \cref{item:SC_full_rank,item:SC_low_rank}, we can also draw a connection to the classic definition of convexity/concavity: $f$ is (strictly) convex/concave if and only if each line $
	H := \{(x,y): \; L_\alpha(x,y) := \alpha_1 + \alpha_2 x + \alpha_3 y = 0\}, \; \alpha \in \mathds{R}^3$
 is crossed at most twice by the curve $\gamma: \{t_1,\dots,t_n\} \to \mathds{R}^2$ with $\gamma(t_i) := (t_i,f(t_i))$, $i \in (1:n)$ (and no more than two points of $\gamma$ can be co-linear in the strict case). In terms of variation bounding, this reads as
\begin{equation}
\forall \alpha \in \mathds{R}^3: \; \vari{L_\alpha(\gamma)} \; (\variz{L_\alpha(\gamma)}) \leq 2 \label{eq:cvx_var},
\end{equation}
which is why \cref{prop:karlin_scm} in case of $m=3$ can be recovered from \cref{lem:convex} by applying the analogous factorization argument as in \cref{eq:fact}.  

\subsection{Convex Contours}
Next, we will extend theses observations to establish our first main result on the characterization of so-called (strictly) convex contours, which we define directly in terms of our variation notions: 
\begin{defn}\label{defn:convc}
	A sequentially $T$-periodic curve $\gamma: \mathds{Z} \to \mathds{R}^2$ is called a \emph{convex contour} ($\gamma \in \cc{T}$) if 
	for all $\alpha \in \mathds{R}^3 \setminus \{0\}$ it holds that $\Sm_c[\alpha_1	\gamma_1+\alpha_2\gamma_2 + \alpha_3] \leq 2$. $\gamma$ is said to be a \emph{strictly convex contour ($\gamma \in \scc{T}$)} if $\Sm_c$ can be replaced by $\Sp_c$. 
\end{defn}
Visually, if $\gamma \in \cc{T}$ does not lie on a line, then the polygon that is traced out by the continuous curve
\begin{equation}
	\mathcal{P}_\gamma(t) := \begin{bmatrix}
		\gamma_1(\floor{t}) + (t-\floor{t})\gamma_1(\ceil{t}) \\
		\gamma_2(\floor{t}) + (t-\floor{t})\gamma_2(\ceil{t})
	\end{bmatrix}, \quad t \in [0,T), \label{eq:convc}
\end{equation} 
i.e., by sequential linear interpolation between $\gamma(0),\gamma(1),\dots, \gamma(T-1),\gamma(0)$, must be the boundary of a convex set in $\mathds{R}^2$ and the curve must be \emph{simple} (i.e., non-intersecting apart from consecutive points $\mathcal{P}_\gamma(t) = \mathcal{P}_\gamma(t+1)$). In the strict case, it is additionally required that no three vertices of the polygon can be co-linear (see~\Cref{fig:cc_ex}). While the primary purpose of $P_\gamma$ lies in geometric visualization, it also allows us to transfer notions such as simple curves, as well as the set enclosed by a curve to the realm of sequences. 
\begin{figure}
	\centering
	\begin{tikzpicture}
		\begin{axis}[xlabel={$\gamma_1(t)$},ylabel={$\gamma_2(t)$},width=12 cm,height = 6 cm]
			
			\addplot[mark=*, mark size=1.5 pt, line width=1 pt,
			red] file{convex_contour_third_order_T_10.txt}; 
			\addplot[draw = none, line width= 1 pt,
			red, postaction={
				decorate,
				decoration={
					markings,
					mark=between positions 0.1 and 1 step 0.15 with {
						\arrow{stealth} %
					}
				}
			}] file{convex_contour_third_order_T_10.txt};
			
\pgfplotstableread[header=false]{convex_contour_third_order_T_10.txt}\datatable
\pgfplotstablegetrowsof{\datatable}
\pgfplotstablegetrowsof{\datatable}
\pgfmathtruncatemacro{\lastrow}{\pgfplotsretval - 2}

\foreach \i in {0,9,8} {
	\pgfplotstablegetelem{\i}{0}\of\datatable
	\let\xval\pgfplotsretval
	\pgfplotstablegetelem{\i}{1}\of\datatable
	\let\yval\pgfplotsretval
	
	\edef\plotcmd{\noexpand\node at (axis cs:\xval,\yval) [inner sep=2pt,right,xshift=1 pt] {$\gamma(\i)$};}
	\plotcmd
}

\foreach \i in {1} {
	\pgfplotstablegetelem{\i}{0}\of\datatable
	\let\xval\pgfplotsretval
	\pgfplotstablegetelem{\i}{1}\of\datatable
	\let\yval\pgfplotsretval
	
	\edef\plotcmd{\noexpand\node at (axis cs:\xval,\yval) [inner sep=2pt,right,yshift=7 pt] {$\gamma(\i)$};}
	\plotcmd
}

\foreach \i in {2} {
	\pgfplotstablegetelem{\i}{0}\of\datatable
	\let\xval\pgfplotsretval
	\pgfplotstablegetelem{\i}{1}\of\datatable
	\let\yval\pgfplotsretval
	
	\edef\plotcmd{\noexpand\node at (axis cs:\xval,\yval) [inner sep=2pt,above,yshift = 3 pt,xshift = -3 pt] {$\gamma(\i)$};}
	\plotcmd
}

\foreach \i in {3} {
	\pgfplotstablegetelem{\i}{0}\of\datatable
	\let\xval\pgfplotsretval
	\pgfplotstablegetelem{\i}{1}\of\datatable
	\let\yval\pgfplotsretval
	
	\edef\plotcmd{\noexpand\node at (axis cs:\xval,\yval) [inner sep=2pt,left,xshift=-3 pt] {$\gamma(\i)$};}
	\plotcmd
}

\foreach \i in {4,5,6} {
	\pgfplotstablegetelem{\i}{0}\of\datatable
	\let\xval\pgfplotsretval
	\pgfplotstablegetelem{\i}{1}\of\datatable
	\let\yval\pgfplotsretval
	
	\edef\plotcmd{\noexpand\node at (axis cs:\xval,\yval) [inner sep=2pt,below,xshift = -4 pt] {$\gamma(\i)$};}
	\plotcmd
}

\foreach \i in {7} {
	\pgfplotstablegetelem{\i}{0}\of\datatable
	\let\xval\pgfplotsretval
	\pgfplotstablegetelem{\i}{1}\of\datatable
	\let\yval\pgfplotsretval
	
	\edef\plotcmd{\noexpand\node at (axis cs:\xval,\yval) [inner sep=2pt,below,xshift = 5 pt] {$\gamma(\i)$};}
	\plotcmd
}

\end{axis}
	\end{tikzpicture}
	\caption{Visualization of a strictly convex contours $\gamma$ with period $T=10$: $\gamma \in \cc{T}$, since the polygon resulting from connecting $\gamma(0),\dots,\gamma(T_1),\gamma(0)$ sequentially is non-intersecting and the boundary of a convex set. Moreover, $\gamma \in \scc{T}$, since no three points $\gamma(t-1),\gamma(t),\gamma(t+1)$ are co-linear. \label{fig:cc_ex}}
\end{figure} 
Using \cref{prop:sc_k_mat_vb_m,lem:convex}, sequentially convex contours can be tractably characterized as follows:
\begin{thm} \label{prop:convc_ct}
	Let $\gamma: \mathds{Z} \to \mathds{R}^2$ be $T$-periodic with $T\geq 4$, 
	\begin{equation}
		M^\gamma := \begin{bmatrix}
			\gamma_1(0) & \gamma_2(0) & 1\\
			\gamma_1(1) & \gamma_2(1) & 1\\
			\vdots & \vdots & \vdots\\
			\gamma_1(T-1) & \gamma_2(T-1) & 1
		\end{bmatrix}, \label{eq:M_mat}
	\end{equation}
	and $\tilde{\gamma}: \mathds{Z} \to \mathds{R}^2$ be the $\tilde{T}$-periodic sequence that results from deleting all points with indices $\mathcal{I}_0 := \{t \in \mathds{Z}: \Delta \gamma(t) = 0\}$ from $\gamma$. If $\rk(M^\gamma) = 3$, then the following are equivalent:
	\begin{thmenum}
		\item $\gamma \in \textnormal{(S)}\cc{T}$. \label{item:prop_CC}
		\item \label{item:prop_compound} $\compound{M^\gamma}{3} \in \textnormal{(S)}\sgc{3}$. 
		\item All determinants \label{item:prop_minor}
		\begin{equation} \label{eq:det_gamma}
			\det\begin{bmatrix}
				 \gamma_1(t_1) & \gamma_2(t_1) & 1\\
				\gamma_1(t_2) & \gamma_2(t_2) & 1\\
				\gamma_1(t_3) & \gamma_2(t_3) & 1
			\end{bmatrix}, t_1 < t_2 < t_3 < t_1 + T
		\end{equation}
		share the same (strict) sign.
		\item \label{item:prop_polygon} $P_\gamma$ is simple such that the set enclosed by $P_\gamma([0,T))$ is convex (and no three points of $\gamma(0:T-1)$ are co-linear). 
		\item  All determinants \label{item:prop_PM}
		\begin{equation} \label{eq:det_gamma_consec}
			\det \begin{bmatrix}
				\Delta \tilde{\gamma}_1(t) & \Delta \tilde{\gamma}_2(t)\\
				\Delta \tilde{\gamma}_1(t+1) & \Delta \tilde{\gamma}_2(t+1) 
			\end{bmatrix}, \; t \in (0,\tilde{T}-1)
		\end{equation}
		share the same (strict) sign and $\tilde{\gamma}_1, \tilde{\gamma}_2 \in \textnormal{(S)}\pmu{\tilde{T}}$.
	\end{thmenum}
If $\rk(M^\gamma) = 2$, then $\gamma \in \cc{T}$ if and only if $\gamma_1, \gamma_2 \in \pmu{T}$. 
\end{thm}
A proof of this result can be found in \Cref{proof:prop:convc_ct}. Note that by applying \cref{item:prop_PM}, the computational burden of checking the (strict) convexity of a curve $\gamma$ is reduced from $\mathcal{O}(T^3)$, required for a naive evaluation of $M^\gamma_{[3]}$, to linear complexity $\mathcal{O}(T)$. Further, the requirement for using $\tilde{\gamma}$ in \cref{item:prop_PM} cannot be removed. An example for this is shown in  \Cref{fig:cc_delete} with $\gamma_1,\gamma_2 \in \pmu{T}$:
\begin{figure}
	\centering
	\begin{tikzpicture}
		\begin{axis}[xlabel={$\gamma_1(t)$},ylabel={$\gamma_2(t)$},width=0.95 \textwidth,height = 4.5 cm, ymin = -0.3, ymax = 1.3, xtick={0,0.1,0.2,...,1}, ytick={0,1}]
			
			\addplot[mark=*, mark size=1.5 pt, line width=1 pt,
			red] coordinates{(0,0) 
				(1,0)
				(1,0)
				(1,1)
				(1,1)
				(0.5,0)
				(0.5,0)};
			\addplot[draw = none, line width= 1 pt,
			red, postaction={
				decorate,
				decoration={
					markings,
					mark=between positions 0.1 and 1 step 0.15 with {
						\arrow{stealth} %
					}
				}
			}] coordinates{
			(0,0) 
			(1,0)
			(1,0)
			(1,1)
			(1,1)
			(0.5,0)
			(0.5,0)}; \label{line:cc_delete_gamma}
			\node[below] at (axis cs:0,0) {$\gamma(0)$};
			\node[below] at (axis cs:1,0) {$\gamma(1) = \gamma(2)$};
			\node[above] at (axis cs:1,1) {$\gamma(3) = \gamma(4)$};
			\node[below] at (axis cs:0.5,0) {$\gamma(5) = \gamma(6)$};
		\end{axis}
	\end{tikzpicture}
	\caption{Illustration for the need of $\tilde{\gamma}$ in \cref{prop:convc_ct}: for the contour $\gamma$ with period $T=7$ all consecutive $3$-minors of $M^\gamma$ in \cref{eq:M_mat} are zero and $\gamma_1, \gamma \in \pmu{T}$. However, the curve resulting from connecting $\gamma(0),\dots,\gamma(T_1),\gamma(0)$ sequentially is intersecting and not the boundary of a convex set, which is why $\gamma \not \in \cc{T}$. \label{fig:cc_delete}}
\end{figure} 
the corresponding determinants defined in \cref{eq:det_gamma_consec} are zero, but $\varic{M^\gamma \alpha} = 4$ for $\alpha = \begin{bmatrix}
	 -1 & 1 & 0.25
\end{bmatrix}^\transp$. After deletion of consecutive repetitive points in $\gamma$, we get a new $4$-periodic $\tilde{\gamma}$ with
\begin{align*}
\begin{bmatrix}
	\tilde{\gamma}_1(0) & \tilde{\gamma}_2(0)\\
	\tilde{\gamma}_1(1) & \tilde{\gamma}_2(1)\\
	\tilde{\gamma}_1(2) & \tilde{\gamma}_2(2)\\
	\tilde{\gamma}_1(3) & \tilde{\gamma}_2(3)\\
			\end{bmatrix} := \begin{bmatrix}
	0 & 0\\
	1 & 0\\
	1  & 1\\
	0.5 & 0\\
	\end{bmatrix}
\end{align*}
where this time
\begin{align*}
	\det \begin{bmatrix}
		\Delta \tilde{\gamma}_1(0) & \Delta \tilde{\gamma}_2(0)\\
		\Delta \tilde{\gamma}_1(1) & \Delta \tilde{\gamma}_2(1) 
	\end{bmatrix} = 1, \; \det \begin{bmatrix}
	\Delta \tilde{\gamma}_1(2) & \Delta \tilde{\gamma}_2(2)\\
	\Delta \tilde{\gamma}_1(3) & \Delta \tilde{\gamma}_2(3) 
\end{bmatrix} = -0.5
\end{align*}
reveals that $\gamma \not \in \cc{7}$. 
\begin{rem}\label{rem:tilde_gamma}
	From the proof of \cref{prop:convc_ct}, we can, further, deduce the following two refinements:
	\begin{remenum}
				\item If $\mathcal{I}_+ := \{i \in (0:T-1): \; \Delta \gamma_1(i) > 0\} = (T_0:T_p)$ and $\mathcal{I}_- := \{ i \in (0:T-1): \; \Delta \gamma_1(i) < 0\} = (T_b:T_n)$, $T_p < T_b$, such that $\vari{\Delta \gamma_2(T_p+1:T_b-1)} = \vari{\Delta \gamma_2(T_p+1:T_b-1)}  = 0$, then $\tilde{\gamma}$ can be replaced by $\gamma$. In particular, this is implied by $\gamma_1 \in \spmu{T}$, because then $T_b \leq T_p + 2$ due to \cref{lem:pmu_delta}. 
		\item The difference between positive and negative signs in \cref{eq:det_gamma} is due to the \emph{orientation} of $\mathcal{P}_\gamma$, i.e., $\mathcal{P}_\gamma$ is \emph{positively (counter-clockwise)/negatively (clock-wise) oriented} if and only if the determinants in \cref{eq:det_gamma} are nonnegative/nonpositive. 
	\end{remenum}
\end{rem}

Finally, by the same factorization as for $m=2$, \cref{prop:convc_ct} provides a direct extension of \cref{prop:karlin_scm} for the case of $m=3$.

\section{Cyclic Variation Bounding Systems} \label{sec:cyc_var_bound}
In this section, we will employ \cref{prop:convc_ct} to derive our tractable certificates for kernels $g \in \ell_\infty(T)$ that define convolution operators $\Conv{g}$ with property $\textnormal{(S)}\pmp{T}$ or $\textnormal{(S)}\cvd{2}(T)$. The latter case will be discussed first. 
\subsection{Cyclic $2$-Variation Diminishing Kernels}
By \cref{eq:cyc_mat_u_y} and \cref{lem:equiv_cvb_cvd}, $\Conv{g} \in \textnormal{(S)}\cvd{2}(T)$ if and only if $\circmat{g} \in \textnormal{(S)}\cvd{2}(T)$. Although this could be checked directly by computing $\compound{\circmat{g}}{3}$ (see~\cref{cor:cvb_k}), using \cref{prop:convc_ct} achieves a significant reduction in complexity. In particular, we want to show that under mild assumptions, it suffices to only check the consecutive $3$-minors of $\circmat{g}$. 

To this end, let $g(0:T-1) \in \mathds{R}^T_{>0} \cup \mathds{R}^T_{< 0}$ and note that for $0 \leq t_1 < t_2 < t_3 \leq T-1$ the factorization 
\begin{equation*}
	\circmat{g}_{(:,\{t_1,t_2,t_3\})}	= \begin{bmatrix}
		g(T-t_3) &  &\\
		& \ddots & \\
		& & g(2T-t_3-1)
	\end{bmatrix} \begin{bmatrix}
		\frac{g(T-t_1)}{g(T-t_3)} & \frac{g(T-t_2)}{g(T-t_3)} &1\\
		\vdots & \vdots & \vdots \\
		\frac{g(2T-t_1-1)}{g(2T-t_3-1)} &  \frac{g(2T-t_2-1)}{g(2T-t_3-1)} & 1
	\end{bmatrix}
\end{equation*}
in conjunction with \cref{eq:Cauchy_Binet,prop:convc_ct} implies that $\circmat{g} \in \textnormal{(S)}\sgc{3}$ if and only if all
 \begin{equation*}
 	\gamma^{l,m}: \mathds{Z} \to \mathds{R}^2, \; t \mapsto \left(\frac{g(t)}{g(t-l)},\frac{g(t-m)}{g(t-l)}\right), \; 2 \leq l \leq T-1, \; 1 \leq m < l
 \end{equation*}
 are (strictly) convex contours. In order to verify this efficiently via \cref{item:prop_PM}, we first want to eliminate the possibility of $\Delta \gamma^{l,m}(t) = 0$, i.e, one can neglect the use of $\tilde{\gamma}^{l,m}$. Since by
  \begin{equation}	\label{eq:sc_circmat_general}
 	\frac{g(t-m)}{g(t-l)} = \prod_{i = 0}^{l-m+1} \frac{g(t-m-i)}{g(t-m-i-1)}, \; m < l,
 \end{equation}
it holds that
 \begin{align*}
 	0 =	\Delta \gamma^{l,m}(t) \; &\Longleftrightarrow \; \frac{g(t-1)}{g(t-l-1)} = \frac{g(t)}{g(t-l)}, \; 	\frac{g(t-1-m)}{g(t-l-1)} = \frac{g(t-m)}{g(t-l)} \\ &\Longleftrightarrow \;  \frac{g(t-l)}{g(t-l-1)} = \frac{g(t-m)}{g(t-m-1)} = \frac{g(t)}{g(t-1)},
 \end{align*}
 this means that we need to add the assumption that $\gamma^{2,1}_2(t) = \frac{g(t-1)}{g(t-2)}$ does not assume any value more than twice within a period. Next, by applying \cref{lem:pmu_mult} to \cref{eq:sc_circmat_general}, we can see that all $\gamma^{l,m}_i(t) \in \textnormal{(S)}\pmu{T}$ if $\gamma^{2,1}_2 \in \textnormal{(S)}\pmu{T}$. In summary, if $\gamma^{2,1}_2 \in \textnormal{(S)}\pmu{T}$ and fulfills our added assumption, then by \cref{eq:det_gamma_consec}, it suffices to check that the consecutive $3$-minors across all $\circmat{g}_{(:,\{t_1,t_2,t_3\})}$ share the same (strict) sign. However, as every consecutive $3$-minor of $\circmat{g}_{(:,\{t_1,t_2,t_3\})}$ equals a consecutive $3$-minor of $\circmat{g}_{(:,(1:3))}$ our reduction claim follows.
 
Finally, let us note that the requirement of $\gamma^{2,1}_2 \in \textnormal{(S)}\pmu{T}$ can also be tied to the consecutive $2$-minors of $\circmat{g}_{(:,(1:3))}$, which in turn can be used to describe its consecutive $3$-minors. Concretely, by \cref{eq:unimod_perido}, $\gamma^{2,1}_2 \in \textnormal{(S)}\pmu{T}$ if and only if there exist $t_1 \leq t_2 \leq t_1+T-1$ such that
\begin{equation*}
	\frac{g(t_1)}{g(t_1-1)} (<) \leq\dots (<) \leq  \frac{g(t_2)}{g(t_2-1)} \geq \frac{g(t_2+1)}{g(t_2)} (>) \geq \dots (>) \geq \frac{g(t_1+T-1)}{g(t_1+T-2)},
\end{equation*}
or equivalently, the sequence of consecutive $2$-minors of $\circmat{g}_{(:,(1:2))}$ 
\begin{equation}
	\compound{{g}}{2}(t) := \det\begin{bmatrix}
		g(t) & g(t-1)\\
		g(t+1) &g(t)
	\end{bmatrix}, \; t \in \mathds{Z}
\end{equation}
fulfills $\varic{\compound{{g}}{2}} \leq 2$ ($\varizc{\compound{{g}}{2}}) \leq 2$). Using the so-called \emph{Dodgson's identity} \cite[0.8.11]{horn2012matrix}, it also holds that all the consecutive $3$-minors of $\circmat{g}$ can be expressed as
\begin{equation*}
	\det\begin{bmatrix}
		g(t) & g(t-1) & g(t-2)\\
		g(t+1) & g(t) & g(t-1)\\
		g(t+2) & g(t+1) & g(t)\\
	\end{bmatrix} g(t) = \det \begin{bmatrix}
		\compound{{g}}{2}(t) & \compound{{g}}{2}(t-1)\\
		\compound{{g}}{2}(t+1) & \compound{{g}}{2}(t)
	\end{bmatrix}.
\end{equation*}
By application of \cref{cor:cvb_k}, we have, therefore, proven the following characterization:
\begin{thm}\label{cor:cvd_2}
	Let $g \in \ell_\infty(T)$, $T \geq 4$, with $g(0:T-1) \in \mathds{R}^T_{>0} \cup \mathds{R}^T_{< 0}$ 
	and $\tilde{g}(t) := \frac{g(t)}{g(t-1)}$, $t \in \mathds{Z}$. Then, if $\rk({\circmat{g}}_{(:,(1:3))}) = 3$ and $\tilde{g}$ does not assume any value more than twice within a period, the following are equivalent:
	\begin{corenum}
		\item $\Conv{g} \in \textnormal{(S)}\cvb{2}(T)$ \label{item:cvd2_conv}
		\item $\circmat{g}_{(:,(1:3))} \in \textnormal{(S)SC}_3$ \label{item:cvd2_sc3}
		\item $\tilde{g} \in \textnormal{(S)}\pmu{T}$ and all \label{item:cvd2_cons_minor} 
	\begin{equation}
		\det\begin{bmatrix}
		g(t) & g(t-1) & g(t-2)\\
		g(t+1) & g(t) & g(t-1)\\
		g(t+2) & g(t+1) & g(t)\\
	\end{bmatrix}, \; 0 \leq t \leq T-3  \label{eq:cvd2_cons_minor}
	\end{equation} 
	share the same (strict) sign. 
	\item The compound sequence \begin{equation} \label{item:cvd2_compound_minors}
		\compound{{g}}{2}(t) := \det\begin{bmatrix}
			g(t) & g(t-1)\\
			g(t+1) &g(t)
		\end{bmatrix}, \; t \in \mathds{Z}
	\end{equation}
	fulfills  ($\varizc{\compound{{g}}{2}}) \leq 2$) $\varic{\compound{{g}}{2}} \leq 2$ and all 
	\begin{equation}
		\det \begin{bmatrix}
			\compound{{g}}{2}(t) & \compound{{g}}{2}(t-1)\\
			\compound{{g}}{2}(t+1) & \compound{{g}}{2}(t)
		\end{bmatrix} \label{eq:cvd2_cons_minor_comp}
	\end{equation}
	share the same (strict) sign. 
	\end{corenum}
\end{thm} 
Note by Remark~\ref{rem:tilde_gamma}, the assumption that $\tilde{g}$ does not assume any value more than twice within a period is already implied by $\tilde{g} \in \spmu{T}$. However, in the non-strict case, this assumption cannot be dropped, since, e.g., for $g \in \pmu{5}$ with $g(t) = 1$ for $t \in (0:3)$ and $g(4) = p \neq 1$ all minors in \cref{eq:cvd2_cons_minor} are nonnegative with
\begin{equation*}
	\det(\circmat{g}_{((1:3),(1:3))}) = \det \begin{bmatrix}
		1 & p & 1\\
		1 & 1 & p\\
		1 &  1 & 1
	\end{bmatrix} = (p - 1)^2 > 0,
\end{equation*}
but the is negative non-consecutive $3$-minor
\begin{equation*}
\det(\circmat{g}_{(\{2,4,5 \},(1:3))}) = \det \begin{bmatrix}
	1 & 1 & p\\
	1 &  1 & 1\\
	p & 1 & 1
\end{bmatrix} = -(p - 1)^2 < 0. 
\end{equation*}
Furthermore, in case that $\rk(\circmat{g}_{(:,(1:3))}) < 3$, it can be shown that $\rk(\circmat{g}) = \rk(\circmat{g}_{(:,(1:3))})$ due to the shift-invariant structure of circulant matrices. Since any circulant matrix is diagonalized by the \textit{Fourier matrix} \cite{davis1979circulant}, its rank is determined by the number of non-zero entries in the \emph{Discrete Fourier Transform} of its generating sequence. Specifically, $\rk(\circmat{g}) = 1$ if and only if $g$ is a constant sequence ($g \equiv a \in \mathds{R}$) and, thus, $\Conv{g}$ is trivially $\cvd{2}$. Similarly, for $T > 2$, $\rk(\circmat{g}) = 2$ is possible if and only if $T$ is even and $g$ is the alternating sequence $g(t) = \alpha_1 + \alpha_2(-1)^t$, $\alpha_1,\alpha_2 \in \mathds{R}$.  However, such a sequence is $2$-periodic, which contradicts our assumption that $T \geq 4$.

Notably, \cref{cor:cvd_2} reduces the verification complexity from $\mathcal{O}(T^5)$ (the cost of a symmetry-aware naive evaluation of $\circmat{g}_{[3]}$) to $\mathcal{O}(T)$. \cref{cor:cvd_2} can also be seen as a cyclic analog of \cref{cor:k_pos_consec_minors} in case of $k=3$. 
\subsubsection{Application to LTI Systems}\label{subsubsec:app_cvd}
We want to apply \cref{cor:cvd_2} to convolution operators $\Conv{g}$ that describe causal DTLTI systems with  transfer function $G(z)$ and minimal realization $(A,b,c,d)$. By \cref{eq:periodic_conv}, we need to apply our result to the periodic summation $g_T$ of the impulse response $g$ (see~\cref{eq:periodic_sum}). While one could directly use \cref{eq:ss_psum} to compute $g_T$, our goal is to express the minors in \cref{eq:cvd2_cons_minor,eq:cvd2_cons_minor_comp} as part of the impulse response of a so-called compound system and demonstrate that the rank condition in \cref{cor:cvd_2} is obsolete. 

We begin by noticing that $\cvd{2}(T)$ is invariant under constant time-shifts, i.e., it suffices to verify this property for $z^{-T} G(z)$. Therefore, it is enough to consider cases, where $G(z)$ has a relative degree of at least $T-1$. By \cref{eq:ss_psum} it holds then that $g_T(t) = c A^{T-1+t}\bar{b}$, $\bar{b} := (I_n - A^T)^{-1}b$, for $t \in (-T+1:T-1)$ and, thus,
\begin{equation*}
	\circmat{g_T}_{(:,(T-2:T))} = \begin{bmatrix}
		cA^{2}\bar{b} & c A \bar{b} & c \bar{b}\\
		cA^{3} \bar{b} & cA^{1} \bar{b} & cA \bar{b} \\
		\vdots & \vdots & \vdots \\
		cA^{T+1} \bar{b} & cA^{T} \bar{b} & cA^{T-1} \bar{b} 
	\end{bmatrix} =
	\mathcal{O}^{T}(A,c) (I_n - A^T)^{-1} \mathcal{C}^3(A,b)_{(:,(3:1))}.
\end{equation*}
This expression provides us with two important insights:
\begin{enumerate}[label=\roman*.)]
	\item $\rk(\circmat{g_T}_{(:,(1:3))}) = \rk(\circmat{g_T}_{(:,(T-2:T))}) = 3$, since $G(z)$ is of order larger than $T$. Thus, the rank assumption in \cref{cor:cvd_2} is fulfilled.
	\item The Hankel matrix ${H_{g_T}(1,T)}_{(:,(1:3))} \in \textnormal{(S)}\sgc{3}(T)$ results from reversing the column order in $\circmat{g_T}_{(:,(T-2:T))}$ and the minors in \cref{eq:cvd2_cons_minor} are the consecutive $3$-minors of $-{H_{g_T}(1,T)}_{(:,(1:3))}$. Hence, it suffices to consider $\det(H_{g_T}(t,3))$ for $t \in (1:T-2)$
\end{enumerate}
By \cref{eq:Cauchy_Binet}, $\det(H_{g_T}(t,3))$, $t \geq 1$, corresponds to the impulse response of the so-called \emph{compound system of order 3} of $(A,\bar{b},c)$ \cite{grussler2020variation}
\begin{equation}
	(\compound{A}{3},\compound{\mathcal{C}^3(A,\bar{b})}{3},\compound{\mathcal{O}^3(A,c)}{3}). \label{eq:compound_3rd}
\end{equation}
 In particular, if \cref{eq:compound_3rd} is (strictly) externally positive/negative (with $\det(H_{g_T}(1,3) \neq 0$), then all consecutive minors in \cref{eq:cvd2_cons_minor} share the same (strict) sign. External positivity of \cref{eq:compound_3rd} can be verified efficiently using certificates such as \cite{grussler2019tractable,taghavian2023external,farina1996existence}. A simple instance of systems $(A,\bar{b},c)$ with this property is the series interconnection of first order lags $\frac{k_i}{z-p_i}$ with $0 \leq p_i < 1$, $k_i > 0$ (see \cite{grussler2020variation}) (with strictness if the compound system \cref{eq:compound_3rd} has at least $3$ non-zero $p_i$)). Moreover, since by \cref{eq:g_T_ss}, the periodic summation of a first order lag $g_i(t) = k_i p_i^{t-1}$ derives as ${g_i}_T(t) = \frac{k_i}{1-p_i^T} p_i^{t-1}$ for $t \in (1:T)$, it follows from  \cref{lem:psum_compo} that if $(A,b,c)$ is the series interconnection of (at least three) first order lags (with non-zero $p_i$), then the same is true for $(A,\bar{b},c)$. 

Finally, under the assumption that $\tilde{g_T}(t) := \frac{g_T(t)}{g_T(t-1)}$ does not assume any value more than twice within a period, we are left by \cref{item:var_varc} with checking that
\begin{equation*}
	\compound{g_T}{2}(t) = \det \begin{bmatrix}
	c A^{t-1}\bar{b} & c A^{t-2}\bar{b}\\
	c A^{t+1}\bar{b} & c A^{t-1}\bar{b}
	\end{bmatrix} = -\compound{\begin{bmatrix}
	c\\
	cA
	\end{bmatrix}}{2} \compound{A}{2}^{t-2} \compound{\begin{bmatrix}
	\bar{b} & A\bar{b}
	\end{bmatrix}}{2}
\end{equation*}
does not change sign more than twice for $t \in (2:T+1)$ (and in the strict case, whenever $\compound{g_T}{2}(t)  = 0$, then $\sign(\compound{g_T}{2}(t-1)) = -\sign(\compound{g_T}{2}(t+1)) \neq 0$). As before, this property holds if it is shared by the impulse response of the \emph{compound system of order $2$} of $(A,\bar{b},c)$:
\begin{equation}
	(\compound{A}{2},\compound{\mathcal{C}^2(A,\bar{b})}{2},\compound{\mathcal{O}^2(A,c)}{2}), \label{eq:compound_2nd}
\end{equation}
which is true, e.g., if \cref{eq:compound_2nd} is (strictly) externally positive/negative (with $	\compound{g_T}{2}(0) \neq 0$). Similar to before, the latter is implied if $(A,b,c)$ is a series interconnection of first order lags $\frac{k_i}{z-p_i}$ with $0 \leq p_i < 1$, $k_i > 0$ (with strictness if $(A,b,c)$ is has at least two non-zero $p_i$). Since our assumption on $\tilde{g}_T$ are trivially fulfilled in these cases, we arrive at the following corollary:
\begin{cor} \label{cor:cvb_2_lags}
	If $G(z) = \prod_{i=1}^n \frac{k_i}{z-p_i}$ with $0 \leq p_i < 1$, $k_i > 0$, 
	then $\Conv{g} \in \textnormal{(S)}\cvb{2}(T)$ for any $T \geq 4$ (with strictness if at least three $p_i$'s are non-zero).
\end{cor}

\subsection{Periodically Monotonicity Preserving Kernels}
\label{subsec:PMP_T}
In the following, we will derive our characterizations of convolution operators $\Conv{g} \in \textnormal{(S)}\pmp{T}$, $g \in \ell_{\infty}(T)$. We begin by noticing that if $\Conv{g} \in \textnormal{(S)}\cvb{2}(T)$, then $\Conv{g} \in \textnormal{(S)}\pmp{T}$: in fact, for any $\gamma \in \mathds{R}$ and $u \in \ell_{\infty}(T)$, it holds that $\Conv{g}u - \gamma \mathbf{1} = \Conv{g}(u - \beta \mathbf{1})$ for some $\beta \in \mathds{R}$. Thus, if $u \in \pmu{T}$, then $\varic{u - \beta \mathbf{1}} \leq 2$, which is why $ \varic{\Conv{g}u - \gamma \mathbf{1}} \leq 2$ (and $\varizc{\Conv{g}u - \gamma \mathbf{1}} \leq 2$ in the strict case) if $\Conv{g} \in \textnormal{(S)}\cvb{2}(T)$, i.e., $\Conv{g} \in \textnormal{(S)}\pmp{T}$. However, by the following equivalences, proven in \Cref{proof:lem:pmp_equiv}, it can be seen that $\cvb{2} \neq \pmp{T}$:
\begin{lem}\label{lem:pmp_equiv}
	For $g \in \ell_{\infty}(T)$, the following are equivalent:
\begin{lemenum}
		\item $\Conv{g} \in \textnormal{(S)}\pmp{T}$. \label{item:pmp_eq}
		\item For all $u \in \ell_\infty(T)$ with $\varic{\Delta u} = 2$, it holds that $\varic{\Conv{g} \Delta u} \leq 2$ (and $\varizc{\Conv{g} \Delta u} \leq 2$ in the strict case). \label{item:delta_pmp_eq}
		\item For all $u \in \ell_\infty(T)$ with $\sum_{i=0}^T u(i) = 0$ and $\varic{u} = 2$, it holds that $\Scm[\Conv{g}u] \leq 2$ (and $\Scp[\Conv{g}u] \leq 2$ in the strict case). \label{item:subspace_pmp_eq}
\end{lemenum}
\end{lem}
In other words, by \cref{item:subspace_pmp_eq}, $\pmp{T}$ does not require $\varic{\Conv{g}u} \leq 2$ for all $u \in \ell_\infty(T)$ with $\varic{u} \leq 2$ and as such implies viewer constraints than $\cvb{2}(T)$. A concrete example of $\Conv{g} \in \pmp{T}$ with $\Conv{g} \not \in \cvb{2}(T)$ is given in \Cref{sec:LTI_ex}. 

Next, we require to introduce the notion of \emph{convexity-preservation}, which by the subsequent lemma is equivalent to $\pmp{T}$ and, thus, allows us to utilize \cref{prop:convc_ct} for our investigations:
\begin{defn}[Convexity Preservation]
   For $g\in \ell_{\infty}(T)$, $\Conv{g}$ is called \emph{(strictly) convexity-preserving} ($\Conv{g} \in \textnormal{(S)}\cp{T}$) if $(\Conv{g} u_1, \Conv{g} u_2) \in \textnormal{(S)}\cc{T}$ for all $u \in \cc{T}$. 
\end{defn}

\begin{lem}\label{lem:PMP_T_CC}
   For $g \in \ell_{\infty}(T)$, $T \geq 4$, the following are equivalent:
\begin{lemenum}
	\item $\Conv{g} \in \textnormal{(S)}\pmp{T}$. \label{item:PMP_T}
	\item $\Conv{g} \in \textnormal{(S)}\cp{T}$.  \label{item:CC}
	\end{lemenum}
\end{lem}
\begin{proof}
				We begin by noticing that $u \in \textnormal{(S)}\cc{T}$ if and only if $\alpha_1 u_1 + \alpha_2 u_2 \in \textnormal{(S)}\pmu{T}$ for all $\alpha \in \mathds{R}^2$. Assuming that $\Conv{g} \in \textnormal{(S)}\cp{T}$, it follows with $\alpha_1 = 1$ and $u_2 \equiv 0$ that $\Conv{g}  u_1 \in \textnormal{(S)}\pmu{T}$ for all $u_1 \in \pmu{T}$. Hence, $\Conv{g} \in \textnormal{(S)}\pmp{T}$. Conversely, if $\Conv{g} \in \textnormal{(S)}\pmp{T}$, then $\alpha_1 \Conv{g} u_1 + \alpha_2 \Conv{g} u_2 = \Conv{g} (\alpha_1 u_1 + \alpha_2 u_2) \in \textnormal{(S)}\pmu{T}$ for all $\alpha \in \mathds{R}^2$ and all $u \in \textnormal{(S)}\cc{T}$, which proves that $\Conv{g} \in \textnormal{(S)}\cc{T}$.
	\end{proof}
In our derivations of \cref{thm:pmp_main}, it suffices to characterize $\textnormal{S}\pmp{T}$, since the non-strict case follows from the following lemma, which is proven in \Cref{proof:lem:strict_pos_limit}:
\begin{lem}\label{lem:strict_pos_limit}
	Let $T \geq 4$ and $G_n(z) = \frac{z^3}{(z-\frac{1}{n+2})^3}$, $n \in \mathds{N}$, with impulse response $g_n$ and corresponding periodic summation. Then, for any $\gamma \in \cc{T}$ it holds that
	\begin{lemenum}
		\item $\lim_{n \to \infty} (\Conv{g_n} \gamma_1,\Conv{g_n} \gamma_2) = \gamma$ \label{item:cc_limit}
		\item $(\Conv{g_n} \gamma_1,\Conv{g_n} \gamma_2) \in \scc{T}$ if at least one determinant of $\gamma$ in \cref{eq:det_gamma} is non-zero. \label{item:cc_conv}
	\end{lemenum}
\end{lem}
We are ready, now, to state our desired tractable characterization of $\textnormal{(S)}\pmp{T}$. 
\begin{thm}\label{thm:pmp_main}
   For $g \in \ell_{\infty}(T)$, $T \geq 4$, the following are equivalent:
    \begin{thmenum}
    	        \item $\Conv{g} \in \textnormal{(S)}\pmp{T}$. \label{item:pmp_theorem}
    	\item $(\Delta g, g) \in \textnormal{(S)}\cc{T}$ with positive orientation.\label{item:PM_g_dg}
    	\item $\Delta g \in \textnormal{(S)PM}(T)$ and $(\Delta g(t))^2 \geq (>) \Delta g(t+1)\Delta g(t-1)$ for all $t \in (0:T-1)$. \label{item:PM_log}      
    \end{thmenum}
\end{thm}
A proof of \cref{thm:pmp_main} is given in \Cref{proof:thm:pmp_main}.
\subsubsection{Application to LTI Systems}
Finally, we want to discuss how \cref{thm:pmp_main} can be applied to convolution operators $\Conv{g}$ that correspond to causal finite-dimensional LTI systems. As in \Cref{subsubsec:app_cvd}, it suffices to consider cases where $G(z)$ has a relative degree of at least $T-1$ with minimal realization $(A,b,c)$. By \cref{eq:ss_psum} it holds then that $\Delta g_T(t) = c A^{T-1+t}\bar{b}$, $\bar{b} := (A-I_n)(I_n - A^T)^{-1}b$, for $t \in (-T+1:T-2)$, which by \cref{eq:Cauchy_Binet} yields that
\begin{align*}
\compound{\Delta g_T}{2}(t) &:=	\det\begin{bmatrix}
			\Delta g_T(t) & \Delta g_T(t-1)\\
			\Delta g_T(t+1) & \Delta g_T(t)
	\end{bmatrix} 
	= -\det \begin{bmatrix}
	 c A^{t-2}\bar{b} & c A^{t-1}\bar{b}\\
	  c A^{t-1}\bar{b} &  c A^{t}\bar{b}
	\end{bmatrix} \\ &= -\compound{\mathcal{O}^2(A,c)}{2} \compound{A}{2}^{t-2} \compound{\mathcal{C}^2(A,\bar{b})}{2}, \; t \in (2:T+1).
\end{align*}
Hence, if $(\compound{A}{2},\compound{\mathcal{C}^2(A,\bar{b})}{2},\compound{\mathcal{O}^2(A,c)}{2})$ -- the compound system of order $2$ to $(A,\bar{b},c)$ -- is (strictly) externally negative (with $\compound{\Delta g_T}{2}(2) > 0$), then the inequality requirement in \cref{item:PM_log} is fulfilled. By analogous arguments as in \Cref{subsubsec:app_cvd}, we can conclude that this is the case if $G(z)$ is the series interconnection of first order lags $\frac{k_i}{z-p_i}$, $k_i > 0$, $p_i \geq 0$ (with at least two non-zero $p_i$ in the strict case). Note that this does not just confirm our previous discussion on $\textnormal{(S)}\cvb{2}(T)$ implying $\textnormal{(S)}\pmp{T}$, but also demonstrates that $\textnormal{S}\pmp{T}$ only requires two of these first order lags vs. three for $\textnormal{S}\cvd{2}$. Using \cref{item:PM_g_dg}, this can also be verified graphically as shown in  \Cref{fig:second_third_order} for 
\begin{equation}
	G_2(z) = \frac{1}{(z-0.7)(z-0.8)} \label{eq:second_order}.
\end{equation}
and
\begin{equation}
	G_3(z) = \frac{1}{(z-0.6)(z-0.7)(z-0.8)} \label{eq:third_order}.
\end{equation}
In particular, since no three points are co-linear for each contour, we can conclude that both systems are, indeed, $\textnormal{S}\pmp{T}$ for all $T$. 
\begin{figure}
	\centering
	\begin{tikzpicture}
		\begin{groupplot}[
			group style={
				group size=2 by 1,
				horizontal sep=2cm,
				horizontal sep=40 pt
			},
			width=0.52 \textwidth,
			height=5.75cm,
			xlabel={$\Delta g_T$},
			ylabel={$g_T$},
			]
			
			\nextgroupplot
			\addplot[dotted, mark=*, mark size=1.5 pt, line width=1 pt,
			blue] file{convex_contour_second_order_T_5.txt}; 
			\label{line:T_5_1}
			\addplot[draw = none, line width= 1 pt,
			blue, postaction={
				decorate,
				decoration={
					markings,
					mark=between positions 0.1 and 0.9 step 0.2 with {
						\arrow{stealth} %
					}
				}
			}, postaction={
				decorate,
				decoration={
					markings,
					mark=at position 0 with {
						\node[below] {$t=0$};
					}
				}
			}] file{convex_contour_second_order_T_5.txt};
			
			\addplot[mark=*, mark size=1.5 pt, line width=1 pt,
			red] file{convex_contour_second_order_T_10.txt}; 
			\label{line:T_10_1}
			\addplot[draw = none, line width= 1 pt,
			red, postaction={
				decorate,
				decoration={
					markings,
					mark=between positions 0.15 and 0.9 step 0.2 with {
						\arrow{stealth} %
					}
				}
			}, postaction={
				decorate,
				decoration={
					markings,
					mark=at position 0 with {
						\node[left] {$t=0$};
					}
				}
			}] file{convex_contour_second_order_T_10.txt};
			
			\addplot[mark=*, mark size=1.5 pt, line width=1 pt,
			color = FigColor3,dashed] file{convex_contour_second_order_T_100.txt}; 
			\label{line:T_100_1};
			\addplot[postaction={
				decorate,
				decoration={
					markings,
					mark=between positions 0.15 and 0.9 step 0.2 with {
						\arrow{stealth} %
					}
				}
			}, postaction={
				decorate,
				decoration={
					markings,
					mark=at position 0 with {
						\node[below] {$t=0$};
					}
				}
			},draw=none,color = FigColor3] table{convex_contour_second_order_T_100_short.txt};
			
			\nextgroupplot
			\addplot[dotted, mark=*, mark size=1.5 pt, line width=1 pt,
			blue] file{convex_contour_third_order_T_5.txt}; 
			\label{line:T_5_2}
			\addplot[draw = none, line width= 1 pt,
			blue, postaction={
				decorate,
				decoration={
					markings,
					mark=between positions 0.1 and 0.9 step 0.2 with {
						\arrow{stealth} %
					}
				}
			}, postaction={
				decorate,
				decoration={
					markings,
					mark=at position 0 with {
						\node[above] {$t=0$};
					}
				}
			}] file{convex_contour_third_order_T_5.txt};
			
			\addplot[mark=*, mark size=1.5 pt, line width=1 pt,
			red] file{convex_contour_third_order_T_10.txt}; 
			\label{line:T_10_2}
			\addplot[draw = none, line width= 1 pt,
			red, postaction={
				decorate,
				decoration={
					markings,
					mark=between positions 0.15 and 0.9 step 0.2 with {
						\arrow{stealth} %
					}
				}
			}, postaction={
				decorate,
				decoration={
					markings,
					mark=at position 0 with {
						\node[right] {$t=0$};
					}
				}
			}] file{convex_contour_third_order_T_10.txt};
			
			\addplot[mark=*, mark size=1.5 pt, line width=1 pt,
			color = FigColor3,dashed] file{convex_contour_third_order_T_100.txt}; 
			\label{line:T_100_2};
			\addplot[postaction={
				decorate,
				decoration={
					markings,
					mark=between positions 0.15 and 0.9 step 0.2 with {
						\arrow{stealth} %
					}
				}
			}, postaction={
				decorate,
				decoration={
					markings,
					mark=at position 0 with {
						\node[below] {$t=0$};
					}
				}
			},draw=none,color = FigColor3] table{convex_contour_third_order_T_100_short.txt};
		\end{groupplot}
	\end{tikzpicture}
	\caption{Visualization of the curve $(\Delta g_T,g_T)$ belonging to $G_2(s)$ in \cref{eq:second_order} (left) and $G_3(s)$ in \cref{eq:third_order} (right) for \ref{line:T_5_1} $T=5$, \ref{line:T_10_1} $T=10$ and \ref{line:T_100_1} $T=100$. All curves are strictly convex contours, which by \cref{thm:pmp_main} implies that $\Conv{g} \in \textnormal{S}\pmp{T}$. \label{fig:second_third_order}}
\end{figure}

\section{Examples -- Relay Feedback Systems} \label{sec:LTI_ex}
In this section, we will further illustrate the application of our results by studying how $\pmp{T}$ affects the existence of self-sustained oscillations in relay feedback systems, i.e., a Lur'e feeback system (see~ \Cref{fig:lure}), where $\psi(\cdot) = \sign (\cdot)$ and self-oscillations are characterized as fixed-points of the loop gain 
\begin{equation}
	u \mapsto -\Conv{g}\sign(u). \label{eq:fixed_point_relay} 
\end{equation}

\subsection{Third Order System -- $\pmp{T}$ for all $T$}
We begin with the case of $G$ being the system in \cref{eq:third_order}. Since we have already verified that $\Conv{g}\in \pmp{T}$ for all $T$ and the same is true for any monotonic static non-linearity, we can conclude that also the open loop gain is $\pmp{T}$ independently of $T$. Thus, the open loop gain leaves the set of periodically monotone signals invariant, which is why one may speculate that \cref{eq:fixed_point_relay} has periodically monotone fixed-points. This is, indeed, the case as visualized in \Cref{fig:third_order_relay}. Note that this fixed-point can be computed and verified via the equation
\begin{equation*}
	u^\ast(0:15) = -\circmat{g_{16}} \begin{bmatrix}
		\phantom{-}	\mathbf{1}_8\\
		-\mathbf{1}_8
	\end{bmatrix}.
\end{equation*}
\begin{figure}
	\centering
	\begin{tikzpicture}
		\begin{axis}[xmin = 0, xmax = 100, grid=both, xlabel = $t$, ylabel= $u^\ast (t)$, width = 0.95\textwidth, height = 4 cm]
			\addplot[ycomb, thin, mark=square*,mark options={fill=red,draw=black},mark size=2pt] file{simluation_third_order_pm_transient.txt};
		\end{axis}
	\end{tikzpicture}
	\caption{Output simulation for relay feedback systems with $G = G_3$ as in \cref{eq:third_order}. After a transient period, the system converges to a periodically monotone self-oscillation of period $T=16$, i.e., a fixed-point to \cref{eq:fixed_point_relay}. \label{fig:third_order_relay}}
\end{figure}

\subsection{Fourth Order System -- $\pmp{T}$ for $T \leq 4$}
In our next example, we use the following fourth order system 
\begin{equation}
	G(z) = \frac{(z-1)(z-3)}{(z-4)^4}. \label{eq:fourth_order}
\end{equation}
Graphically (see~\Cref{fig:fourth_order_cc}), we can verify that $\Conv{g} \in \pmp{T}$ only for $T \in (2:4)$, which is why this time one may only expect that there exists a $\pmu{T}$ fixed-point of the loop-gain for $T \in (2:4)$. By computing

 \begin{figure}
	\centering
	\begin{tikzpicture}
		\begin{axis}[xlabel={$\Delta g_T$},ylabel={$g_T$},width=0.95 \textwidth,height =  5.5cm]
			\addplot[blue, mark = *,mark size=1.5 pt, line width=1 pt]file{convex_contour_fourth_order_T_4.txt}; \label{line:T_4_4}
				\addplot[draw = none, line width= 1 pt,
			blue, postaction={
				decorate,
				decoration={
					markings,
					mark=between positions 0.1 and 0.9 step 0.2 with {
						\arrow{stealth} %
					}
				}
			}, postaction={
				decorate,
				decoration={
					markings,
					mark=at position 0 with {
						\node[right] {$t=0$};
					}
				}
			}] file{convex_contour_fourth_order_T_4.txt};

			\addplot[red,dotted,mark = *,mark size=1.5 pt, line width=1 pt] file{convex_contour_fourth_order_T_5.txt}; \label{line:T_4_5}
						\addplot[draw = none, line width= 1 pt,
			red, postaction={
				decorate,
				decoration={
					markings,
					mark=between positions 0.1 and 0.9 step 0.2 with {
						\arrow{stealth} %
					}
				}
			}, postaction={
				decorate,
				decoration={
					markings,
					mark=at position 0 with {
						\node[left] {$t=0$};
					}
				}
			}] file{convex_contour_fourth_order_T_5.txt};

			\addplot[dashed,color= FigColor3,mark = *,mark size=1.5 pt, line width=1 pt] file{convex_contour_fourth_order_T_10.txt}; \label{line:T_4_10}
									\addplot[draw = none, line width= 1 pt,
			color = FigColor3, postaction={
				decorate,
				decoration={
					markings,
					mark=between positions 0.1 and 0.9 step 0.2 with {
						\arrow{stealth} %
					}
				}
			}, postaction={
				decorate,
				decoration={
					markings,
					mark=at position 0 with {
						\node[left] {$t=0$};
					}
				}
			}] file{convex_contour_fourth_order_T_10.txt};

		\end{axis}
	\end{tikzpicture}
	\caption{Visualization of the curve $(\Delta g_T,g_T)$ belonging to $G$ in \cref{eq:fourth_order} with \ref{line:T_4_4} $T=4$, \ref{line:T_4_5} $T=5$ and \ref{line:T_4_10} $T=10$. Since only the curve corresponding to $T=4$ is a convex contour, it follows by \cref{thm:pmp_main} that $\Conv{g} \in \pmp{T}$ exclusively for $T = 4$. \label{fig:fourth_order_cc}
	}
\end{figure}

\begin{equation*}
	u^\ast(0:4) = -\circmat{g_{4}} \begin{bmatrix}
	\phantom{-}	\mathbf{1}_2\\
		-\mathbf{1}_2
	\end{bmatrix} = \begin{bmatrix}
		\phantom{-}	4.5261 \\ \phantom{-} 1.2672\\ -4.5261 \\ -1.2672
	\end{bmatrix},
\end{equation*}
we can see that $u^\ast \in \pmu{4}$ corresponds to such a self-oscillation. Thus, revealing the significance of the $\pmp{T}$ property for the analysis of self-oscillations in relay feedback systems. Further, note that as
\begin{equation*}
	\compound{{{\circmat{g_{4}}}_{(:,(1:3))}}}{3} = 5.7506 \begin{bmatrix}
		1 & -1 & 1 &-1
	\end{bmatrix}^\transp,
\end{equation*}
$\Conv{g} \not \in \cvb{2}$. This confirms our analysis in \cref{lem:pmp_equiv} and shows the practical need for the less restrictive notion of $\pmp{T}$.  

\section{Conclusion}\label{sec:conc}
In this work, we have provided tractable characterizations for two nested classes of cyclic variation bounding discrete-time LTI systems: periodic monotonicity preservation (PMP) and cyclic $2$-variation diminishing. Central to our derivations is a bridge between algebraic total positivity theory \cite{karlin1968total} and its geometric interpretations, which has led to a tractable characterization of sequentially convex contours. 

Our results are intended to facilitate the development of new fixed-point theorems that will support the study of self-sustained oscillations in Lur'e feedback systems. The presented examples on relay feedback systems provide a first indication of how the PMP property is important for the existence of such oscillations by limiting the potential set of oscillations and periods. A similar supporting observation has been made recently in \cite{chaffey2024amplitude,tong2025selfsustained}. %
In the future, we would also like to extend these results to signals with more sign changes period and complete our preliminary studies in \cite{tong2025selfsustained} with general fixed-point theorems.

\bibliographystyle{plain}
\bibliography{refkpos,refopt,refpos,science}
\appendix \label{sec:appendix}
\crefalias{Section}{Appendix}
\crefalias{section}{appendix}
\section{Proof to \cref{lem:pmu_mult}}
\label{proof:lem:pmu_mult}
	Let $u \in \spmu{T}$, where after a possible time-shift, it can be assumed that $$u(0) < u(1) < \dots < u(t^\ast) \geq u(t^\ast +1) > u(t^\ast +2) > \dots > u(T-1) > u(0).$$ Then, $\tilde{u} \in \spmu{T}$, since
	\begin{align*}
		u(0)u(1)  < \dots < u(t^\ast - 1)u(t^\ast)\text{ and }
		 u(t^\ast) u(t^\ast+1)>   \dots > u(T-1)u(0)
	\end{align*}
	with the only equality possible at $u(t^\ast - 1)u(t^\ast) = u(t^\ast) u(t^\ast+1)$. The non-strict case follows analogously. 
\section{Proof to \cref{lem:convex}}
\label{proof:lem:convex}
	The equivalence between \cref{item:cvx_f,item:bar_f} follows as discussed prior to \cref{lem:convex} by \cite[Proposition~1.1.4 \& 5.3.1]{hiriart2013convex}. The additional requirements in the strict case are readily verified by noticing that three points $(t_{i_1},f(t_{i_1}))$, $(t_{i_2},f(t_{i_2}))$, $(t_{i_3},f(t_{i_3}))$ are co-linear if and only if $s(t_{i_1},t_{i_2}) = s(t_{i_2},t_{i_3})$. Finally, to see the equivalence with \cref{item:cvx_det,item:cvx_comp}, it suffices to note that the elements in $M_{[3]}^f$ compute as
	\begin{align*}
		\det \begin{bmatrix}
			t_{i_1} & f(t_{i_1}) & 1\\
			t_{i_2} & f(t_{i_2}) & 1\\
			t_{i_3} & f(t_{i_3}) & 1
		\end{bmatrix} 
		&= \det \begin{bmatrix}
			t_1& f(t_1) & 1\\
			t_2 - t_1 & f(t_2) - f(t_1) & 0\\
			t_3 - t_2 &f(t_3) - f(t_2) & 0
		\end{bmatrix} = \det \begin{bmatrix}
			1 & \frac{f(t_{i_2}) - f(t_{i_1})}{t_{i_2} - t_{i_1}}\\
			1 & \frac{f(t_{i_3}) - f(t_{i_1})}{t_{i_3} - t_{i_1}}
		\end{bmatrix}\\
			&= s(t_{i_2},t_{i_3}) - s(t_{i_1},t_{i_2})
		\end{align*}
		for $i_1 < i_2 < i_3$. Hence, \cref{item:cvx_comp} is equivalent to \cref{def:cvx} and by \cite[Proposition~5.3.1]{hiriart2013convex} it is enough to check the consecutive determinants in \cref{item:cvx_det}.
	\section{Proof to \cref{prop:convc_ct}}
	\label{proof:prop:convc_ct}
		In the following let $L_\alpha(x,y) := \alpha_1 x + \alpha_2 y + \alpha_3$ for any $\alpha \in \mathds{R}^3 \setminus \{0\}$ and $\rk(M^\gamma) = 3$. We begin by noticing that as $
			\Sm_c[L_\alpha(\gamma)] = \Sm_c\left[ M^\gamma \alpha\right]$ and $\Sp_c[L_\alpha(\gamma)] = \Sp_c\left[ M^\gamma \alpha\right]$,
			it follows from \cref{item:CVB_SC} (\cref{item:SCVB_SSC}) that $\gamma \in \textnormal{(S)}\cc{T}$ if and only if $M^\gamma \in \textnormal{(S)}\sgc{3}$ or, equivalently all determinants of the form \cref{eq:det_gamma} share the same (strict) sign. This proves the equivalence between \cref{item:prop_CC,item:prop_compound,item:prop_minor}. 
			
			\cref{item:prop_CC} $\Rightarrow$ \cref{item:prop_PM}: Note that as \cref{eq:det_gamma_consec} are merely $3$-minors of $M^\gamma$, it follows via \cref{item:prop_compound} that these determinants share the same (strict) sign. Moreover, as $\Sm_c[L_\alpha\tilde{\gamma})] (\Sp_c[L_\alpha(\tilde{\gamma})]) \leq 2$ whenever $\alpha_1 = 0$ or $\alpha_2 = 0$, it must hold that $\tilde{\gamma}_1, \tilde{\gamma}_2 \in \textnormal{(S)}\pmu{T}$.
			
			\cref{item:prop_PM}  $\Rightarrow$ \cref{item:prop_polygon}: We begin by remarking that the requirements on $P_\gamma$ are implied by verifying them for $P_{\tilde{\gamma}}$. Therefore, we can assume in the following that $\mathcal{I}_0 = \emptyset$. Further, by our rank assumption, $\gamma_1$ cannot be constant, which is why
			\begin{align*}
				\mathcal{I}_+ := \{ i \in (0:T-1): \; \Delta \gamma_1(i) > 0\}  \quad \text{and} \quad
				\mathcal{I}_- := \{ i \in (0:T-1): \; \Delta \gamma_1(i) < 0\}  
			\end{align*}
			are non-empty sets. Assuming, w.l.o.g., that $0 \in \mathcal{I}_+$, $T-1 \not \in \mathcal{I}_+$  and that all determinants in \cref{eq:det_gamma_consec} are (strictly) positive, we want to show, now, that there exits a $T_p \in (1:T-2)$ such that $\mathcal{I}_+ = (0:T_p)$. Since by assumption $\gamma_1 \in \pmu{T}$, it follows from \cref{lem:pmu_delta} that if no such $T_p$ existed, then we could find  $0 \leq t_1 < t_2 < T- 1$ such that $\Delta \gamma_1(t_1) > 0, \; \Delta \gamma_1(t_1+1) = 0, \; \Delta \gamma_1(t_2) > 0$. 
			To see that this is impossible, note that $\Delta \gamma_2(t_1+1) \neq 0$ by assumption. Thus, an evaluation of \cref{eq:det_gamma_consec} at $t=t_1$ and $t = t_1+1$ yields
			\begin{equation}
				\Delta \gamma_1(t_1) \Delta \gamma_2(t_1+1) > 0 \quad \text{and} \quad -\Delta \gamma_1(t_1+2) \Delta \gamma_2(t_1+1) \geq 0, \label{eq:eval}
			\end{equation}
			and as such $\Delta \gamma_1(t_1+2) = 0$. Inductive application of this argument shows that $\Delta \gamma_1(t_2) = 0$. Analogously, we can show that there exist
			$T_b,T_n \in (T_p+1:T-1)$ such that $\mathcal{I}_- = (T_b:T_n)$. Thus, by \cref{lem:convex}  
			\begin{equation*}
				f:  \{\gamma_1(0),\dots,\gamma_1(T_p+1)\} \to \mathds{R}, \; \gamma_1(i) \mapsto \gamma_2(i)
			\end{equation*}
			is (strictly) convex and
			\begin{equation*}
				g:\{\gamma_1(T_b),\dots, \gamma_1(T_n+1)\}\to \mathds{R}, \; \gamma_1(i) \mapsto \gamma_2(i).
			\end{equation*}
			(strictly) concave. Next, it follows analogously to \cref{eq:eval} that 
			\begin{align*}
				\Delta \gamma_2(T_p+1), \Delta \gamma_2(T_b-1) > 0 \quad \text{and} \quad
				\Delta \gamma_2(T_n+1),  \Delta \gamma_2(T-1) < 0.
			\end{align*}
			Thus, by $\gamma_2 \in \pmu{T}$ and \cref{lem:pmu_delta}, we can conclude that  $\Delta \gamma_2(t) \geq 0$ for $t \in (T_p+1:T_b-1)$ and $\Delta \gamma_2(t) \leq 0$ for $t \in (T_n+1:T-1)$. Hence, $P_\gamma$ is simple and such that $g$ lies above the graph of $f$ and
			\begin{align*}
				\gamma_1(T_p+1) = \dots = \gamma_1(T_b) \quad \text{ and }
				\gamma_1(T_n) = \dots = \gamma_1(T) = \gamma_1(0)
			\end{align*}
			Therefore, by \cref{lem:convex}, $P_\gamma([0,T)])$ is the intersection of the epi-graph of the continuous convex function $\bar{f}$ and the hyper-graph of a continuous concave function $\bar{g}$ as defined in \cref{eq:f_ext}. As these sets are convex \cite[Proposition~IV.1.1.6]{hiriart2013convex}, the same applies to their intersection \cite[Proposition~III.1.2.1]{hiriart2013convex}. 
			In particular, if there exist three co-linear points of $\gamma(0:T-1)$, the convexity of this set implies that also all points with intermediate indices have to lie on the same line, i.e., there exist three consecutive points of $\gamma(0:T-1)$ that have to be co-linear. In the strict case, this is impossible on the sets $(0:T_p+1)$ and $(T_b:T_n+1)$ due to the strict convexity and concavity of $f$ and $g$, respectively (see \cref{lem:convex}). Also on the intervals that include the remaining points, we can exclude this possibility as $\gamma_1 \in \spmu{T}$ requires that $\Delta \gamma_1$ has no consecutive zeros, i.e., $T_b - T_p, T-T_n \leq 2$ by \cref{lem:pmu_delta}.

			\cref{item:prop_polygon} $\Rightarrow$ \cref{item:prop_minor}: By convexity and simplicity, there is no line that can intersect $P_\gamma([0,T))$ more than twice, i.e., $\Sm_c[L_\alpha(\gamma)] =  \Sm_c \left[M^\gamma \alpha\right] \leq 2$ for all $\alpha \in \mathds{R}^3 \setminus \{0\}$ or equivalently $M^\gamma$ is $\cvb{2}$. By \cref{item:CVB_SC}, this requires that all minors in \cref{eq:det_gamma} share the same sign, where the sign is strict if no three points of $\gamma(0:-T-1)$ are co-linear.

			In case that $\rk(M^\gamma) = 2$, $\gamma$ lies on the line segment between two points, $\gamma(t_{\min})$ and $\gamma(t_{\max})$, where $\gamma_i(t_{\min}),\gamma_i(t_{\max}) \in \max{\gamma_i(0:T-1)} \cup  \min{\gamma_i(0:T-1)}$ for $i \in \{1,2\}$.  Hence,  $\Sm_c[L_\alpha(\gamma)]  \leq 2$ for all $\alpha \in \mathds{R}^3$ if and only if $\Delta \gamma_i$ is monotonically increasing/decreasing when traversing from $\gamma(t_{\min})$ to $\gamma(t_{\max})$ and monotonically decreasing/increasing on the way back. In other words, $\Sm_c(\Delta \gamma_i) \leq 2$ for $i \in \{1,2\}$, which by \cref{lem:pmu_delta} is equivalent to $\gamma_1,\gamma_2 \in \pmu{T}$.

		\section{Proof to \cref{lem:pmp_equiv}}
		\label{proof:lem:pmp_equiv}
			\cref{item:pmp_eq} $\Leftrightarrow$ \cref{item:delta_pmp_eq}: Note that $u \in \ell_{\infty}(T)$ satisfying $\Delta u \not \equiv 0$ and  $\varic{\Delta u} \leq 2$ are given by $u \in \ell_\infty(T)$ satisfying $\varic{\Delta u} = 2$. Moreover, since $\Delta (\Conv{g}u) = \Conv{g} (\Delta u)$, it holds that $\varic{\Delta (\Conv{g}u)} = \varic{\Conv{g} (\Delta u)}$
			which is why the equivalence is implied by \cref{lem:pmu_delta}. \\
			\cref{item:subspace_pmp_eq} $\Leftrightarrow$ \cref{item:delta_pmp_eq}: 
				Since $\Delta$ is an invertible operator on the set of $u \in \ell_\infty(T)$ satisfying $\sum_{i=0}^T u(i) = 0$, it follows that for every such $u$ there exists a $\tilde{u} \in \ell_\infty(T)$ with $u = \Delta \tilde{u}$ and vice-versa. Thus, the equivalence follows.

			\section{Proof to \cref{lem:strict_pos_limit}}
			\label{proof:lem:strict_pos_limit}
			By \cref{cor:cvb_2_lags} and the shift-invariance of $\svd{2}(T)$, it holds that $\Conv{{{g}_n}_T} \in \svd{2}(T)$. Therefore, by our discussion in the beginning of \Cref{subsec:PMP_T}, it holds that $\Conv{{g_n}_T} \in \textnormal{S}\pmp{T}$, which by \cref{lem:PMP_T_CC} implies that $\Conv{{g_n}_T} \in \textnormal{S}\cp{T}$. In particular, as $\lim_{n \to \infty}G_n(z) = 1$, \cref{item:cc_limit} follows. Finally, since by \cref{cor:cvd_2}, $\circmat{{g_n}_T} \in \ssgc{3}$ and
			\small{\begin{align*}
			\circmat{{g_n}_T} \begin{bmatrix}
				\gamma_1(0) & \gamma_2(0) & 1\\
				\gamma_1(1) & \gamma_2(1) & 1\\
				\vdots & \vdots & \vdots\\
				\gamma_1(T-1) & \gamma_2(T-1) & 1
			\end{bmatrix} = \begin{bmatrix}
			(\Conv{{g_n}_T} \gamma_1)(0:T-1) & (\Conv{{g_n}_T} \gamma_2)(0:T-1) & p \mathbf{1}_{T} 
			\end{bmatrix}
			\end{align*}}
		\normalsize
			for some $p \in \mathds{R}$, \cref{item:cc_conv} follows from \cref{eq:Cauchy_Binet} and \cref{prop:convc_ct}.

			\section{Proof to \cref{thm:pmp_main}}
			\label{proof:thm:pmp_main}
				In case that $\Delta g = 0$, i.e., $g$ is constant, each item holds and as such they are equivalent. Therefore, we can assume that $\Delta g \not \equiv 0$, which, since $\sum_{t=0}^{T-1} \Delta g(t) = 0$, implies that $\Delta g$ has strictly positive and strictly negative elements.

			\cref{item:pmp_theorem} $\Rightarrow$ \cref{item:PM_g_dg}: Let $u \in \cc{T}$ be such that $u(kT) = e_1$, $u(1+kT) = e_2$, $k \in \mathds{Z}$, and otherwise $u(t)$ is zero.  Since by \cref{lem:PMP_T_CC}, $\Conv{g} \in \textnormal{(S)}\cp{T}$, it holds then that $(\Conv{g} u_1, \Conv{g} u_2) \in \textnormal{(S)}\cc{T}$, where $(\Conv{g} u_1)(t) = g(t)$, $(\Conv{g} u_2)(t) = g(t-1)$, it follows from \cref{prop:convc_ct} that all non-vanishing determinants
				\begin{equation*}
					\det\begin{bmatrix}
						g(t_1) & g(t_1-1) & 1\\
						g(t_2) & g(t_2-1) & 1 \\
						g(t_3) & g(t_3-1) & 1
					\end{bmatrix}, \; t_1 < t_2 < t_3 < t_1 + T
				\end{equation*}
				have the same (strict) sign. Equivalently, the same holds true for all 
				\begin{equation} \label{eq:det_delta_g}
					\det\begin{bmatrix}
						\Delta g(t_1-1) & g(t_1-1) & 1 \\
						\Delta g(t_2-1) & g(t_2-1) & 1 \\
						\Delta g(t_3-1) & g(t_3-1) & 1
					\end{bmatrix}, \; t_1 < t_2 < t_3 < t_1 + T
				\end{equation}
				which by \cref{prop:convc_ct} shows that $(\Delta g,g) \in \textnormal{(S)}\cc{T}$. Thus, it must hold that $\Scm(\Delta g)= 2$, i.e., there exists a $t$ such that $\Delta g(t) \neq 0$ and $\Delta g(t-1) \Delta g(t+1) \leq 0$. Hence,  
				\begin{equation}
					\begin{aligned}
							0 &<  \det \begin{bmatrix}
							\Delta g(t) & \Delta g(t-1)\\
							\Delta g(t+1) & \Delta g(t)
						\end{bmatrix} = \det\begin{bmatrix}
							\Delta g(t-1) & g(t-1) & 1 \\
							\Delta g(t) - \Delta g(t-1) & \Delta g(t-1) & 0\\
							\Delta g(t+1) - \Delta g(t) & \Delta g(t) & 0
						\end{bmatrix}\\ & = \det\begin{bmatrix}
							\Delta g(t-1) & g(t-1) & 1\\
							\Delta g(t) & g(t) & 1\\
							\Delta g(t+1) & g(t+1) & 1
						\end{bmatrix},
					\end{aligned}
				 \label{eq:dg_log}
				\end{equation} 
				i.e., $(\Delta g,g)$ has positive orientation. 
				
				\cref{item:PM_g_dg} $\Rightarrow$ \cref{item:PM_log}: From \cref{prop:convc_ct} and \cref{eq:dg_log}, we receive the inequality conditions as well as that $ \Delta g \in \textnormal{(S)}\pmu{T}$. 
				
				\cref{item:PM_log}  $\Rightarrow$ \cref{item:PM_g_dg}: Since by assumption $\Scm(\Delta g)= 2$ and $\Delta g \in \textnormal{(S)}\pmu{T}$, it holds that $g \in \textnormal{(S)}\pmu{T}$ and each of the sets 
				\begin{align*}
					\mathcal{I}_+ &:= \{ i \in (0:T-1): \; \Delta g(i) > 0\} \\
					\mathcal{I}_- &:= \{ i \in (0:T-1): \; \Delta g(i) < 0\}
				\end{align*}
				must consist of consecutive entries. Therefore, the requirements of Remark~\ref{rem:tilde_gamma} are fulfilled for $\gamma_1 := g$ and $\Delta \gamma_2 :=  g$ and \cref{prop:convc_ct} can be applied even when replacing $\tilde{\gamma}$ with $\gamma$, which in conjunction with \cref{eq:dg_log} proves the claim.  
				
				\cref{item:PM_g_dg} $\Rightarrow$ \cref{item:pmp_theorem}: By \cref{lem:strict_pos_limit,lem:lim_var}, it suffices to prove the strict case, which by \cref{lem:pmp_equiv} and \cref{item:var_varc} is equivalent to showing that $\Delta y := \Conv{g} \Delta u$ fulfills $\varizc{\Delta y}\leq 2$ for all $u \in \ell_{\infty}(T)$ with $\Scm[\Delta u] = 2$. To this end, we assume, w.l.o.g., that there exists an $s \in (0:T-3)$ such that
				$\Delta u(0),\dots,\Delta u(s) \geq 0$ with $\Delta u(0) > 0$ and $\Delta u(s+1),\dots,\Delta u(T-1) \leq 0$ with $\Delta u(s+1) < 0$ (analogous arguments also apply if $s = T-2$). Combining this assumption with the fact that $\Delta {u}(T-1) = -\sum_{i=0}^{T-2} \Delta u(i)$ allows us to express $\Delta y(k:k+2) = v_1 - v_2$ with
				\begin{align*}
					v_1 &:=  \sum_{i=0}^{s} |\Delta u(i)| \begin{bmatrix}
						g(k-i)-g(k-1)\\
						g(k+1-i)-g(k)\\
						g(k+2-i)-g(k+1)
					\end{bmatrix}\\
					v_2 &:= \sum_{j=s+1}^{T-2} |\Delta u(j)| \begin{bmatrix}
						g(k-j)-g(k-1)\\
						g(k+1-j)-g(k)\\
						g(k+2-j)-g(k+1)
					\end{bmatrix}.
				\end{align*}
				By the multi-linearity of the determinant, we arrive then at
				\begin{align*}
					0 &=\det \begin{bmatrix}
						\Delta y(k:k+2) & v_1 &v_2
					\end{bmatrix} 
					=\sum_{i=1}^s |\Delta u(i)| \sum_{j=s+1}^{T-1} |\Delta u(j)|	\det(M_{k,i,j})
				\end{align*}
				for
				\begin{align*}
					\small{M_{k,i,j} := \begin{bmatrix}
							\Delta y(k) & g(k-i)-g(k-1) &  	g(k-j)-g(k-1)\\
							\Delta y(k+1) & g(k+1-i)-g(k) & g(k+1-j)-g(k) \\
							\Delta y(k+2) & g(k+2-i)-g(k+1)&  g(k+2-j)-g(k+1)
					\end{bmatrix}},
				\end{align*}
				which is why by Laplace expansion of $\det(M_{k,i,j})$
				\begin{equation}
					0 = \sum_{i=1}^s |\Delta u(i)| \sum_{j=s+1}^{T-1} |\Delta u(j)| \left[ \Delta y(k) \alpha_{k,i,j} - 	\Delta y(k+1)  \alpha_{k+1,i,j}+ \Delta y(k+2) \alpha_{k+2,i,j}\right] \label{eq:contradiction}
				\end{equation} with
				\begin{align*}
					\alpha_{k+2,i,j} &= \det \begin{bmatrix}
						g(k-i+1) - g(k+1)& g(k-j+1) - g(k+1)\\
						g(k-i+2) - g(k+2)& g(k-j+2) - g(k+2)
					\end{bmatrix} \\ &= \det\begin{bmatrix}
						\Delta g(k+1) &  g(k+1)    & 1  \\
						\Delta g(k-i+1) &  g(k-i+1) & 1\\
						\Delta g(k-j+1)  &  g(k-j+1) &1
					\end{bmatrix} > 0\\
					\alpha_{k+1,i,j} & =  \det \begin{bmatrix}
						g(k-i) - g(k)& g(k-j) - g(k)\\
						g(k-i+2) - g(k+2)& g(k-j+2) - g(k+2)
					\end{bmatrix} \\
					\alpha_{k,i,j} &=  \det \begin{bmatrix}
						g(k-i) - g(k)& g(k-j) - g(k)\\
						g(k-i+1) - g(k+1)& g(k-j+1) - g(k+1)
					\end{bmatrix} \\ &= \det\begin{bmatrix}
						\Delta g(k) &  g(k)    & 1 \\
						\Delta g(k-i) &  g(k-i) & 1\\
						\Delta g(k-j)  &  g(k-j) & 1
					\end{bmatrix} > 0.
				\end{align*}	
				Since $\Delta u(0) \neq 0 \neq \Delta u (s+1)$, it follows that if $\Delta {y}(k+1) = 0$, then \cref{eq:contradiction} requires that 
				\begin{equation}
					\Delta {y}(k)\Delta {y}(k+2) < 0  \label{eq:delta_y_prop}
				\end{equation}
				or $\Delta {y}(k) = \Delta {y}(k+1) = \Delta {y}(k+) = 0$. In the first case, $\Scm[\Delta y] = \Scp[\Delta y]$, while in latter case it follows inductively that $\Delta {y} \equiv 0$. Note that $\Delta y \equiv 0$ if and only if $v_1 = v_2$, in which case \cref{eq:contradiction} has to hold even after substitution with $\Delta y(k) = 1$, $\Delta y(k+1) = 0$ and $\Delta y(k+2) = 1$, which, as just confirmed, is impossible. 
				
				Finally, we will show that if $\Scm[\Delta y] > 2$, then it is possible to construct a $\hat{u} \in \pmu{T}$ such that $\Delta \hat{y} = \Conv{g} \Delta \hat{u}$ does not fulfill \cref{eq:delta_y_prop}. To this end, note that since $\Delta g \in \spmu{T}$ by \cref{item:PM_log}, it holds by assumption that $\Delta g(i) = 0$ if and only if $\Delta g(i-1) \Delta g(i+1) < 0$, i.e., $\Scm(\Delta g) = \Scp(\Delta g) = 2$. Thus, if $\Scm(\Delta y) > 2$, there must exists a $t \in (0,1)$ and $k^\ast$ such that $\Delta \hat{y}: = (1-t)\Delta g + t\Delta y$, $\Delta \hat{y}(k^\ast) = 0$ and $\Delta \hat{y}(k^\ast-1)\Delta \hat{y}(k^\ast+1) > 0$. However, as $\Delta g = \Conv{g}  \Delta u_1$ for $u_1(0:T-1) := e_1$, it follows that $\Delta \hat{y} = \Conv{g}\Delta \hat{u}$ with $\Delta \hat{u} = (1-t)\Delta u_1 + t\Delta u$, which by assumption on $\Delta u(0) > 0$ and $\Delta u(T-1) \leq 0$ fulfills $\Scm(\Delta \hat{u}) \leq 2$. This yields the desired contradiction.

\end{document}